\documentclass[10pt]{amsart}
\theoremstyle{plain}
\usepackage{amsmath, amsfonts, amsthm, amscd, mathrsfs}
\usepackage{graphics}
\usepackage{enumerate}
\usepackage{color}
\usepackage{epsfig}
\usepackage[all]{xy}
\usepackage{graphicx,amssymb}

\graphicspath{ {Figures/} }
\numberwithin{figure}{section}

\theoremstyle{theorem}

\newtheorem{theorem}{Theorem}[section]

\newtheorem{cor}[theorem]{Corollary}
\newtheorem{definition}[theorem]{Definition}
\newtheorem{notation}[theorem]{Notation}

\newtheorem{lemma}[theorem]{Lemma}
\newtheorem{prop}[theorem]{Proposition}
\newtheorem{claim}[theorem]{Claim}
\newtheorem{remark}[theorem]{Remark}
\newtheorem{assumption}[theorem]{Assumption}

\newcommand{\Addresses}{{
  \bigskip
  \footnotesize

  James Farre, \textsc{Department of Mathematics, Yale University}\par\nopagebreak
  \textit{E-mail address}: \texttt{james.farre@yale.edu}
  }}

\newcommand{\G}{\PSL_2 \mathbb{C}}
\newcommand{\g}{\Gamma}
\newcommand{\ie}{{\itshape i.e.} }
\renewcommand{\to}[1][]{\xrightarrow{\ #1\ }}

\newcommand{\bR}{\mathbb{R}}
\newcommand{\bRplus}{\mathbb{R}^{\ge 0}}

\newcommand{\bC}{\mathbb{C}}

\newcommand{\bU}{\mathbb{U}}
\newcommand{\bS}{\mathbb{S}}

\newcommand{\bZ}{\mathbb{Z}}
\newcommand{\bN}{\mathbb{N}}
\newcommand{\bH}{\mathbb{H}}

\newcommand{\cH}{\mathcal{H}}
\newcommand{\cEL}{\mathcal{EL}}
\newcommand{\QQ}{\mathcal{Q}}
\newcommand{\cPML}{\mathcal{PML}}
\newcommand{\teich}{\mathscr{T}}

\newcommand{\cN}{\mathcal{N}}

\newcommand{\inverse}{^{-1}}
\newcommand{\Hbb}{\Hb_{\bb}}
\newcommand{\Cbb}{\Cb_{\bb}}

\newcommand{\model}[1]{\mathfrak{#1}}
\newcommand{\ladder}{\mathcal L}

\DeclareMathOperator{\cl}{cl }
\DeclareMathOperator{\Isom}{Isom}

\DeclareMathOperator{\Hb}{H}
\DeclareMathOperator{\bb}{b}
\DeclareMathOperator{\Cb}{C}

\DeclareMathOperator{\PSL}{PSL}

\DeclareMathOperator{\vol}{vol}
\DeclareMathOperator{\Vol}{Vol}
\DeclareMathOperator{\area}{area}
\DeclareMathOperator{\im}{im}

\DeclareMathOperator{\inj}{inj}
\DeclareMathOperator{\str}{str}

\DeclareMathOperator{\supp}{supp}

\DeclareMathOperator{\bilip}{bilip}
\DeclareMathOperator{\core}{core}
\DeclareMathOperator{\Jac}{Jac}
\DeclareMathOperator{\sech}{sech}
\DeclareMathOperator{\pleat}{pleat}
\DeclareMathOperator{\Hom}{Hom}

\begin{document}
\title{Relations in Bounded Cohomology}
\author{James Farre}
\maketitle
\begin{abstract}
We explain some interesting relations in the degree three bounded cohomology of surface groups.  Specifically, we show that if two faithful Kleinian surface group representations are quasi-isometric, then their bounded fundamental classes are the same in bounded cohomology.  This is novel in the setting that one end is degenerate, while the other end is geometrically finite.  We also show that a difference of two singly degenerate classes with bounded geometry is boundedly cohomologous to a doubly degenerate class, which has a nice geometric interpretation.  Finally, we explain that the above relations completely describe the linear dependences between the `geometric'  bounded classes defined by the volume cocycle with bounded geometry.  We obtain a mapping class group invariant Banach sub-space of the reduced degree three bounded cohomology with explicit topological generating set and describe all linear relations.   
\end{abstract}

\section{Introduction}
The cohomology of a surface group with negative Euler characterisic is well understood.  In contrast, the bounded cohomology in degree $1$ vanishes, in degree $2$, it is an infinite dimensional Banach space with the $\|\cdot\|_\infty$ norm \cite{mm:banach, ivanov:banach}, and in degree $3$ it is infinite dimensional but not even a Banach space \cite{soma:nonBanach}.  In degree $4$ and higher, almost nothing is known (see Remark \ref{remark n=4}).  In this paper, we study a subspace in degree $3$ generated by \emph{bounded fundamental classes} of infinite volume hyperbolic $3$-manifolds homotopy equivalent to a closed oriented surface $S$ with negative Euler characteristic.  These manifolds correspond  to $\G$ conjugacy classes of discrete and faithful representations $\rho: \pi_1(S)\to \G$, but we restrict ourselves to the representations that do not contain parabolic elements to avoid technical headaches.  Algebraically, the \emph{bounded fundamental class} of a manifold will be the pullback, via $\rho$, of the \emph{volume class} $\Vol\in \Hb_{\text{cb}}^3(\G;\bR)$.  It can also be understood as the singular bounded cohomology class with representative defined by taking the signed volume of a straightened tetrahedron  (see Sections \ref{bounded spaces}-\ref{bounded fundamental class}).  When we restrict our attention to bounded fundamental classes with bounded geometry, we will actually give a complete list of the linear dependences among these bounded classes.  As a consequence we obtain a subspace of $\Hb_{\bb}^3(\pi_1(S);\bR)$ on which the seminorm $\|\cdot \|_\infty$ restricts to a norm; this gives us a Banach space and we provide an explicit uncountable topological basis (see Section \ref{relations and banach subspace}).  This subspace is also mapping class group invariant and any linear dependencies among bounded classes have a very nice geometric description in terms of a certain cut-and-paste operation on hyperbolic manifolds (see Theorem \ref{relation 1} and the discussion following it).  

The positive resolution of Thurston's Ending Lamination Conjecture for surface groups due to \cite{minsky:ELTI} and \cite{brock-canary-minsky:ELTII} building on work of \cite{masur-minsky:complex1} and \cite{masur-minsky:complex2} tells us that the \emph{end invariants} of a hyperbolic structure on $S\times \bR$ determine its isometry class which, in the totally degenerate setting, is a version of the motto, ``topology implies geometry.''  Suppose $\rho_1$ and $\rho_2$ are discrete, faithful, without parabolics, and their quotient manifolds $M_{\rho_i}$ share one geometrically infinite end invariant.  Then they are quasiconformally conjugate, and there is a bi-Lipschitz homeomorphism in the preferred homotopy class of mappings $M_{\rho_1}\to M_{\rho_2}$ inducing $\rho_2\circ\rho_1\inverse$ on fundamental groups (see Section \ref{singly degenerate classes}).  The bi-Lipschitz constant depends on the dilatation of the quasi-conformal conjugacy, which is essentially the exponential of the Teichm\"uller distance between their geometrically finite end invariants.  This bi-Lipschitz homeomorphism lifts to covers, inducing an equivariant quasi-isometry.  Our first main result in this paper is

\begin{theorem}\label{singly degenerate}
If $\rho_1$ and $\rho_2$ are discrete, faithful, without parabolics, and their quotient manifolds $M_{\rho_i}$ are singly degenerate and share one geometrically infinite end invariant, then $\rho_1^*\Vol = \rho_2^*\Vol\in \Hb_{\bb}^3(\pi_1(S);\bR)$.
\end{theorem}
We emphasize that Theorem \ref{singly degenerate} holds even for manifolds with unbounded geometry.  In Sections \ref{compactly supported classes}-\ref{zero out half the manifold}, we restrict ourselves to the setting of manifolds with \emph{bounded geometry}.  A manifold $M$ has bounded geometry if its injectivity radius $\inj(M)$ is  positive.  That is, there is no sequence of essential closed curves whose lengths tend to $0$.  An ending lamination $\lambda\in \cEL(S)$ has \emph{bounded geometry} if any singly degenerate manifold with $\lambda$ as an end invariant has bounded geometry.  Note that if some singly degenerate manifold with geometrically infinite end invariant $\lambda$ and no parabolic cusps has bounded geometry, then every singly degenerate manifold with $\lambda$ as an end invariant and no cusps has bounded geometry \cite[Bounded geometry theorem]{Minsky:bounded_geometry}.  This is closely related to the fact that any Teichm\"uller geodesic ray tending toward a bounded geometry ending lamination $\lambda$ stays in a compact set \cite{Rafi:bdd}.  Let $\cEL_\text{b}(S)\subset \cEL(S)$ be the laminations that have bounded geometry.  Let $\lambda, \lambda' \in \cEL_{\text b}(S)$ and $X, Y\in \teich (S)$, then we have representations $\rho_{(\nu_{\scriptscriptstyle-},\nu_+)}: \pi_1(S)\to \G$ with end invariants $\nu_{\scriptscriptstyle-}, \nu_+ \in\{\lambda, \lambda', X, Y\}$.  Let $\widehat{\omega}({\nu_{\scriptscriptstyle-}},{\nu_+}) \in \Cb_{\bb}^3(S;\bR)$ be the corresponding bounded volume $3$-cocycle (see Section \ref{bounded fundamental class}).  
In Section \ref{relations and banach subspace} we prove our main results relating doubly degenerate bounded classes to each other.  The following theorem states that the doubly degenerate bounded classes decompose into a sum of singly degenerate bounded classes.  
\begin{theorem}\label{relation 1}
Let $\lambda, \lambda' \in \cEL_{\text b}(S)$ and $X, Y\in \teich (S)$ be arbitrary.  We have an equality in bounded cohomology \[[\widehat{\omega}(\lambda', \lambda)]= [\widehat{\omega}(\lambda',X)]+[\widehat{\omega}(Y, \lambda)] \in \Hb^3_{\bb}(S;\bR).\]
\end{theorem}

We can think of the singly degenerate bounded classes as ``atomic.''  If we cut a doubly degenerate manifold with end invariants $(\lambda', \lambda)$ along an embedded surface, we are left with two manifolds, each of which is bi-Lipschitz equivalent to the convex core of a hyperbolic manifold with end invariants $(\lambda', X)$ or $(X, \lambda)$.  This gives a geometric explanation for Theorem \ref{relation 1}, once we establish that bounded cohomology  ignores bounded geometric perturbations.   As a corollary, we see that the ``cohomological shadows'' of geometrically infinite ends vanish under addition in $\Hb_{\bb}^3(S;\bR)$.
\begin{cor}\label{main cor}
Suppose $\lambda_1, \lambda_2, \lambda_3 \in \cEL_{\text b}(S)$ are distinct. Then we have an equality in bounded cohomology \[ [\widehat{\omega}(\lambda_1, \lambda_2)]+ [\widehat{\omega}(\lambda_2,\lambda_3 )] = [\widehat{\omega}(\lambda_1, \lambda_3)] \in \Hb^3_{\bb}(S;\bR).\]
\end{cor}
We now state some consequences of the results of this paper and the classification theory for finitely generated Kleinian groups.  To do so, we restate results of previous work of the author in the case of marked Kleinian surface groups.
\begin{theorem}[{\cite[Theorem 6.2 and Theorem 7.7]{Farre:bdd}}]\label{old work}
Fix a closed, orientable surface $S$ of negative Euler characteristic.  There is an $\epsilon_S>0$ such that if $\{[\rho_\alpha]\}_{\alpha\in \Lambda}\subset \Hom(\pi_1(S), \G)/\G$ are discrete, faithful, and without parabolics such that at least one of the geometrically infinite end invariants of $\rho_\alpha$ is different from the geometrically infinite end invariants of $\rho_\beta$ for all $\beta\not= \alpha \in \Lambda$, then 
\begin{enumerate}
\item $\{\rho_\alpha^*\Vol\}_{\alpha\in \Lambda}\subset \Hb_{\bb}^3(\pi_1(S);\bR)$ is a linearly independent set.  
\item $\|\sum_{i = 1}^N a_i\rho_{\alpha_i}^*\Vol\|_\infty \ge \epsilon_S\max{|a_i|}$.
\end{enumerate}
\end{theorem}

Say that $\rho_1, \rho_2 : \g \to \G$ are \emph{quasi-isometric} if there exists a $(\rho_1, \rho_2)$-equivariant quasi-isometry $\bH^3\to \bH^3$.  Combining Theorems \ref{singly degenerate} and \ref{old work} with the Ending Lamination Theorem (Theorem \ref{ELT}) and the quasi-conformal deformation theory (see the discussion following Theorem \ref{main}), we obtain
\begin{cor}\label{qi_invariance}
The bounded fundamental class is a quasi-isometry invariant of discrete and faithful representations of $\pi_1(S)$ without parabolics.  In this setting, $\|\rho_0^*\Vol - \rho_1^*\Vol\|_\infty<\epsilon_S$ if an only if $\rho_0$ is quasi-isometric to $\rho_1$.  
\end{cor}
Soma \cite[Theorem A and Corollary B]{Soma:ELT} has recently obtained a version of Corollary \ref{qi_invariance} that does not factor through the curve complex machinery that was used in the proof of the Ending Lamination Theorem and instead relies on the almost rigidity properties of hyperbolic tetrahedra with almost maximal volume.  Subsequently, Soma provides an alternate approach to the Ending Lamination Theorem \cite[Theorem C and Corollary D]{Soma:ELT}.  

It turns out that the seminorm on bounded cohomology can also detect representations with dense image and faithful representations in the following sense:
\begin{theorem}[{\cite[Theorem 1.1]{Farre:dense},\cite[Theorem 1.3]{Farre:bdd}}]\label{d_and_f}
Suppose $\rho_0: \pi_1(S) \to \G$ is discrete, faithful, has no parabolic elements, and at least one geometrically infinite end invariant and let $\rho: \pi_1(S) \to \G$ be arbitrary.  If $\|\rho_0^*\Vol - \rho^*\Vol\|_\infty<\epsilon_S/2$, then $\rho$ is faithful.  If $\rho$ has dense image, then $\|\rho_0 ^*\Vol - \rho^*\Vol\|_\infty \ge v_3$, where $v_3$ is the volume of the regular ideal tetrahedron in $\bH^3$.  
\end{theorem}
We then obtain the following rigidity theorem as a corollary of Theorem \ref{d_and_f} and Corollary \ref{qi_invariance}.  
\begin{cor}
Suppose $\rho_0: \pi_1(S)\to \G$ is discrete and faithful with no parabolics and at least one geometrically infinite end invariant and $\rho_1: \pi_1(S)\to \G$ has no parabolics (but is otherwise arbitrary).  If $\|\rho_0^*\Vol - \rho_1^*\Vol\|_\infty <\epsilon_S/2$, then $\rho_1$ is discrete and faithful, hence quasi-isometric to $\rho_0$.  
\end{cor}

Let $M$ be a topologically tame hyperbolic $3$-manifold with incompressible boundary (that is, the compact $3$-manifold whose interior is homeomorphic to $M$ has incompressible boundary) and no parabolic cusps.  The inclusion of any surface subgroup $i_S: \pi_1(S)\to \pi_1(M) = \g\le \G$ corresponding to an end of $M$  induces a seminorm non-increasing map $i_S^*:\Hb_{\bb}^3(\pi_1(M);\bR) \to \Hb_{\bb}^3(\pi_1(S);\bR)$.  If the end corresponding to $S$ is geometrically infinite and $M$ is not diffeomorphic to $S\times \bR$, by the Covering Theorem, $i_S: \pi_1(S)\to \G$ is a singly degenerate marked Kleinian surface group.  By Theorem \ref{singly degenerate}, $i_S^*\Vol$ identifies the geometrically infinite end invariant of $M_{i_S}$ (equivalently, the end invariant of $M$ corresponding to $S$). Say that a hyperbolic manifold is \emph{totally degenerate} if all of its ends are geometrically infinite.  Applying Waldhausen's Homeomorphism Theorem \cite{Waldhausen:homeo} and the Ending Lamination Theorem \cite{brock-canary-minsky:ELTII}, we have
\begin{theorem}
Suppose $M_0$ and $M_1$ are hyperbolic $3$-manifolds without parabolic cusps with holonomy representations $\rho_i: \pi_1(M_i)\to \G$, $M_1$ is totally degenerate, and $h: M_1\to M_2$ is a homotopy equivalence.  If $M_1$ is topologically tame and has incompressible boundary, then there is an $\epsilon$ depending only on the topology of $M_1$ such that $h$ is homotopic to an isometry if and only if  $\|\rho_0^*\Vol - \rho_1^*\Vol\|_\infty < \epsilon$.
\end{theorem}

Apply Theorem \ref{old work} to see that the singly degenerate classes form a linearly independent set, and they are uniformly separated from each other in seminorm.  Fix a base point $X\in \teich(S)$, and define $\iota: \cEL(S) \to \Hb_{\bb}^3(S;\bR)$ by the rule $\iota(\lambda) = [\widehat{\omega}(X, \lambda)]$.  By Theorem \ref{singly degenerate}, $\iota$ does not depend on the choice of $X$, and $\iota$ is mapping class group equivariant.  We summarize here, and elaborate in Section \ref{relations and banach subspace}. 
\begin{theorem}
The image of $\iota$ is a linearly independent set.  Moreover, for all $\lambda, \lambda' \in \cEL(S)$, if $\|\iota(\lambda) - \iota(\lambda')\|_\infty <\epsilon_S$ then $\lambda = \lambda'$.  Finally, $\iota$ is mapping class group equivariant, and $\iota(\cEL(S))$ is a topological basis for the image of its closure in the reduced space $\overline{\Hb}^3_{\bb}(S;\bR)$.  
\end{theorem}
 
By Corollary \ref{main cor}, we know that the $\bR$-span of $\iota(\cEL_{\bb}(S))$ contains all bounded classes of doubly degenerate manifolds with bounded geometry.  The results of this paper give a complete characterization of the linear dependencies among elements in the closure of this Banach subspace.  Again, see Section \ref{relations and banach subspace}.  

\begin{remark}\label{remark n=4}The bounded fundamental class is a construction that, for surface and free groups, is necessarily uninteresting in dimension $4$ (more generally even dimensions at least $4$).  Let $S$ be a compact, oriented surface and $\rho: \pi_1(S) \to \Isom^+(\bH^n )$ be discrete and faithful with $n\ge 4$ even.  Then \cite{Bowen:cheeger} shows that the hyperbolic $n$-manifold $\bH^n/\im\rho$ has positive Cheeger constant.  Moreover, \cite{Kim-Kim:cheeger} show that positivity of the Cheeger constant is equivalent to the vanishing of the bounded class $\rho^*\Vol_n \in \Hb_{\bb}^n(S;\bR)$.     We arrive at the claim in the beginning of the remark.
\end{remark}

The organization of the paper is as follows.  In Section \ref{background} we review the definitions of (continuous) bounded cohomology of groups and spaces, some terminology from Kleinian groups, the singular Sol metric on the universal bundle over a Teichm\"uller geodesic as a model for bounded geometry manifolds, and notions in coarse geometry.  In Section \ref{singly degenerate classes}, we consider singly degenerate classes and prove Theorem \ref{singly degenerate}.  We exploit Geometric Inflexibility Theorems to obtain volume preserving, bi-Lipchitz maps $\Phi: M_0\to M_1$ between singly degenerate manifolds that share their geometrically infinity end invariant.  Essentially, we use this map to compare $\Phi(\str_0(\tau))$ with $\str_1(\Phi(\tau))$, where $\tau: \Delta_2 \to M_0$.  We find a homotopy between these two maps with bounded volume, which allows us to express the difference in the bounded fundamental classes of the two manifolds as a bounded coboundary.  The proof of Theorem \ref{relation 1} is modeled on the same strategy, but we need to make a few technical detours and assume that our manifolds have bounded geometry.  Namely, we will need to take a limit of bi-Lipschitz, volume preserving maps to obtain a volume preserving map (up to `compact error') from the convex core of a singly degenerate manifold to a doubly degenerate manifold.  
The assumption that our limit manifold has bounded geometry allows us to ensure that the convex core boundaries of manifolds further out in the sequence get (linearly) further away from some fixed reference point.  We use this to get control over the bi-Lipschitz constants of our maps using geometric inflexibility.  Without the bounded geometry assumption, it is possible (generic) that we cannot take bounded quasi-conformal jumps toward some ending lamination $\lambda$ while making uniform progress away from our reference point, because there are large subsurface product regions where distance from the convex core boundary grows only logarithmically instead of linearly.    We extract the volume preserving limit map in Section \ref{volume preserving limit maps}.

Since our map is only volume preserving away from a compact subset of the convex core of our singly degenerate manifold (that is, up to `compact error'), we need to see that bounded cohomology does not witness this compact error.  We consider functions $f: M\to \bR$ that have compact support, and show that when we weight the hyperbolic volume by $f$, that this new bounded  class is indistinguishable from the old in bounded cohomology.  This we consider in Section \ref{compactly supported classes}.  

In Section \ref{ladders}, we take a coarse geometric viewpoint.  We study \emph{hyperbolic ladders} in the singular Sol metric on the universal bundle over a Teichm\"uller geodesic. This viewpoint was inspired by \cite{Mahan:CT-trees}, and we use it to understand the behavior of geodesics straightened with respect to two `nested metrics'.  Essentially, we can choose the geometrically finite end invariant of a singly degenerate manifold so that there is a $B$-bi-Lipschitz embedding of its convex core into the doubly degenerate manifold, allowing us to think of one as a subspace of the other with the path metric.  We will straighten a based geodesic loop with respect to both metrics and use the geometry of ladders to show that the two straightenings coarsely agree, when they can.  That is, the two paths will fellow travel when it is most efficient to travel in the subspace, and when it is not, one geodesic will stay close to the boundary of the subspace while the other finds a shorter path.  We observe that thin triangles mostly track their edges, so to understand where two geodesic triangles live inside our manifolds, it suffices to understand the trajectories of their edges.  We will use these observations to `zero out' half of a doubly degenerate manifold with a smooth bump function (on an entire geometrically infinite end)  and prove that the resulting bounded class is boundedly cohomologous to that of the singly degenerate class in Section \ref{zero out half the manifold}.  We reiterate that the scheme for the proof there is based on that in Section  \ref{singly degenerate classes}.  

Finally, we prove our main results in Section \ref{relations and banach subspace}; they now follow somewhat easily from the work in previous sections.  
Throughout the paper, we reserve the right to use several different notations where convenient to hopefully improve the exposition.
\section*{Acknowledgements}
The author would like to thank the anonymous referee for their careful reading and suggestions to make this manuscript more readable.  I would like to thank Maria Beatrice Pozzetti for many very enjoyable and useful conversations as well as her interest in this work.  I am also very grateful for the very speedy responses from Mahan Mj and for the introduction to hyperbolic ladders that lead to the contents of Section \ref{ladders}.  I would like to thank Mladen Bestvina for `liking questions,' Yair Minsky for valuable conversations that lead to the content of Section \ref{compactly supported classes}, and for Ken Bromberg's patience, availability, and guidance.  Finally, I acknowledge the support of the NSF, in particular, grants DMS-1246989, DMS–1509171, and DMS-1902896.

\section {Background}\label{background}
\subsection{Bounded Cohomology of Spaces}\label{bounded spaces}
Given a connected CW-complex $X$, we define a norm on the singular chain complex of $X$ as follows.  Let $\Sigma_n = \lbrace \sigma: \Delta_n\to X\rbrace$ be the collection of singular $n$-simplices.  Write a simplicial chain $A\in \Cb_n\left( X;\bR\right)$ as an $\bR$-linear combination \[A = \sum \alpha_\sigma\sigma,\] where each ${\sigma\in \Sigma_n}$.  The $1$-norm or \emph{Gromov} norm of $A$ is defined as \[\left\| A \right\|_ 1 = \sum \left| \alpha_\sigma\right|.\]
This  norm promotes the algebraic chain complex $\Cb_\bullet (X ;\bR)$ to a chain complex of normed linear spaces; the boundary operator is a bounded linear operator.  Keeping track of this additional structure, we can take the topological dual chain complex
\[\left(\Cb_\bullet (X;\bR),\partial, \| \cdot  \|_1\right)^*=\left(\Cb^\bullet_{\bb} (X;\bR),d, \| \cdot \|_\infty\right).\]
   The $\infty$-norm is naturally dual to the $1$-norm, so the dual chain complex consists of \emph{bounded} co-chains.  
Define the \emph{bounded cohomology} $\Hb_{\bb}^\bullet (X ;\bR)$ as the (co)-homology of this complex.  For any bounded $n$-co-chain, $\alpha\in \Cb_{\bb}^n(X;\bR)$, we have an equality \[\|\alpha\|_\infty = \sup_{\sigma\in \Sigma_n}\left|\alpha ( \sigma) \right|.\] The $\infty$-norm descends to a pseudo-norm on the level of bounded cohomology.  If $A\in \Hb_{\bb}^n(X;\bR)$ is a bounded class, the seminorm is described by \[\| A \|_\infty = \inf_{\alpha\in A} \|\alpha\|_\infty.\]We direct the reader to \cite{gromov:bdd} for a systematic treatment of bounded cohomology of topological spaces and  fundamental results.  

Matsumoto-Morita \cite{mm:banach} and Ivanov  \cite{ivanov:banach} prove independently that in degree 2, $\|  \cdot \|_\infty$ defines a norm in bounded cohomology, so that the space $\Hb_{\bb}^2(X;\bR)$ is a Banach space with respect to this norm.  In \cite{soma:nonBanach}, Soma shows that the pseudo-norm is in general not a norm in degree $\ge 3$.  In Section \ref{relations and banach subspace}, we will consider the quotient $\overline{\Hb}_{\bb}^3(S;\bR) =\Hb_{\bb}^3(S;\bR)/Z$ where $Z\subset \Hb_{\bb}^3(S;\bR)$ is the subspace of zero-seminorm elements.  Then $\overline{\Hb}_{\bb}^3(S;\bR)$ is a Banach space with the $\| \cdot \|_\infty$ norm.   

 \subsection{Continuous Bounded Cohomology of Groups}\label{bounded groups}
Let $G$ be a topological group.  We define a co-chain complex for $G$ by considering the collection of continuous, $G$-invariant functions \[\Cb^n(G;\bR) = \lbrace G^{n+1} \to \bR\rbrace.\]  The homogeneous coboundary operator $d$ for the trivial $G$ action on $\bR$ is, for ${f\in \Cb^n(G;\bR)}$, \[d f (g_0, ... ,g_{n+1}) = \sum_{i=0}^{n+1} (-1)^i f(g_0, ..., \hat{g}_i,..., g_{n+1}),\] where $\hat g_i$ means to omit that element, as usual.    
The coboundary operator gives the collection $\Cb^\bullet(G;\bR)$ the structure of a (co)-chain complex.  An $n$-co-chain $f$ is \emph{bounded} if \[\|f\|_\infty = \sup|f(g_0, ..., g_{n})| < \infty,\] where the supremum is taken over all $n+1$ tuples $(g_0, ..., g_{n}) \in G^{n+1}$.  

The operator  ${d : \Cb^n_{\bb}(G;\bR) \to \Cb^{n+1}_{\bb}(G;\bR)}$ is a bounded linear operator with operator norm at most $n+2$, so the collection of bounded co-chains $\Cb^\bullet_{\bb}(G;\bR)$ forms a subcomplex of the ordinary co-chain complex. The cohomology of $(\Cb_{\bb}^\bullet(G;\bR),d)$ is called the \emph{continuous bounded cohomology} of $G$, and we denote it $\Hb_{\text{cb}}^\bullet (G;\bR)$.  When $G$ is a discrete group, the continuity assumption is vacuous, and we write $\Hb_{\text{b}}^\bullet (G;\bR)$ to denote the bounded cohomology of $G$ in the case that it is discrete.  The $\infty$-norm $\|\cdot\|_\infty$ descends to a pseudo-norm on bounded cohomology in the usual way.  A continuous group homomorphism $\varphi: H\to G$ induces a map $\varphi^*:\Hb_{\text{cb}}^\bullet (G;\bR)\to \Hb_{\text{cb}}^\bullet (H;\bR)$ that is norm non-increasing.  Brooks \cite{brooks:remarks}, Gromov \cite{gromov:bdd}, and Ivanov \cite{ivanov:isometric} proved the remarkable fact that for any connected CW-complex $M$, the classifying map $K(\pi_1(M),1)\to M$ induces an isometric isomorphism $\Hb^\bullet_{\bb}(M;\bR)\to \Hb^\bullet_{\bb}(\pi_1 (M);\bR)$.  We therefore identify the two spaces $\Hb^\bullet_{\bb}(\pi_1(M);\bR)= \Hb^\bullet_{\bb}(M;\bR)$.  

\subsection{The Bounded Fundamental Class}\label{bounded fundamental class}
Let $x\in \bH^3$ and consider the function $\vol_x: (\G)^4\to \bR$ which assigns to $(g_0, ..., g_3)$ the signed hyperbolic volume of the convex hull of the points $g_0x, ..., g_3x$.  Any geodesic tetrahedron in $\bH^3$ is contained in an ideal geodesic tetrahedron.  There is an upper bound $v_3$ on volume that is maximized by a \emph{regular} ideal geodesic tetrahedron \cite{Thurston:notes}, so $\displaystyle \| \vol_x\|_\infty = v_3$.
One checks that $d\vol_x=0$, so that $[\vol_x]\in \Hb_{\text{cb}}^3(\G;\bR)$.  Moreover, for any $x,y\in \bH^3$, $[\vol_x] = [\vol_y]\not=0$.  Define $\Vol = [\vol_x]$; the continuous bounded cohomology $\Hb_{\text{cb}}^3(\G;\bR)= \langle \Vol\rangle_\bR$, and in fact $\|\Vol\|_\infty = v_3$, as well (see e.g. \cite{BBI:mostow} for a discussion of the hyperbolic volume class in dimensions $n\ge 3$).
Let $\rho: \g \to \G$ be any group homomorphism.  Then $\rho^*\Vol\in\Hb_{\text{cb}}^3(\g;\bR)$ is called the \emph{bounded fundamental class of $\rho$}.  

We now specialize to the case that $\g = \pi_1(S)$, where $S$ is a closed oriented surface of negative Euler characteristic and give a geometric description of the bounded fundamental class.  If $\rho: \pi_1(S)\to \G$ is discrete and faithful, then the quotient $M_\rho = \bH^3/\im\rho$ is a hyperbolic manifold and it comes equipped with a homotopy equivalence $f: S \to M_\rho$ inducing $\rho$ on fundamental groups.  

 Let $\omega\in \Omega^3(M_\rho)$ be such that $\pi^*\omega$ is the Riemannian volume form on $\bH^3$ under the covering projection $\bH^3 \to[\pi] M_\rho$.  Suppose $\sigma: \Delta_3\to M_\rho$ is a singular 3-simplex.  We have a chain map  \cite{Thurston:notes}
\[\str : \Cb_\bullet (M_\rho) \to \Cb_\bullet (M_\rho)\] defined by homotoping $\sigma$,  relative to its vertex set, to the unique totally geodesic hyperbolic tetrahedron $\str \sigma$.  The co-chain \[ \widehat\omega(\sigma)= \int_{\str\sigma}\omega \equiv \int_{\Delta_3}(\str\sigma)^*\omega\] measures the signed hyperbolic volume of the straitening of $\sigma$.  
We use the fact that $\str$ is a chain map, together with Stokes' Theorem to observe that if $\upsilon: \Delta_4 \to M_\rho$ is any singular 4-simplex, 
\[d \widehat\omega(\upsilon) = \int_{\str  \partial  \upsilon  } \omega= \int_{\partial \str \upsilon} \omega = \int_{\str \upsilon} d\omega = 0,\]
because $d\omega \in \Omega^4(M_\rho)= \{0\}$.  The class $[\widehat\omega]\in \Hb^3_{\bb}(M_\rho;\bR)$ is the \emph{bounded fundamental class} of $M_\rho$.  Under the isometric isomorphism $f^*: \Hb_{\bb}^3(M_\rho;\bR) \to \Hb_{\bb}^3(\pi_1(S);\bR)$ induced by $f: S \to M_\rho$, $f^*[\widehat\omega] = \rho^*\Vol$.  If $M_\rho$ has end invariants $\nu = (\nu_-, \nu_+)$ (see sections \ref{ends 1} - \ref{infiniteends}), we also use the notation $[\widehat\omega(\nu_-, \nu_+)]$ to denote the bounded fundamental class.

\subsection{Coarse geometry}
A general reference for material in this section is \cite{BH}.  
Let $(X,d_X)$ and $(Y,d_Y)$ be metric spaces (whenever possible we use subscripts to explain which metric we are calculating distances with respect to).  A $(\lambda, \epsilon)$-\emph{quasi-isometric embedding} is a not necessarily continuous map $f: X\to Y$ such that, for all $x,y\in X$, 
\[ \lambda\inverse d_X(x,y) -\epsilon \le d_Y(f(x), f(y))\le \lambda d_X(x,y) +\epsilon.\]
A $(\lambda,\epsilon)$-\emph{quasi-geodesic} is a quasi-isometric embedding of $\bZ$ or $\bR$.  
A $(\lambda, \epsilon)$-qi-embedding is a \emph{quasi-isometry} if it is coarsely surjective.  That is, for all $y\in Y$ there is an $x\in X$ such that $d_Y(f(x), y)<\epsilon$.  

Let $(X, d)$ be a geodesic metric space, and for $x,y \in X$, denote by $[x,y]$ a geodesic segment (we will often conflate geodesics as maps parameterized by or proportionally to arc-length and their images) joining $x$ and $y$.   If there is some $\delta\ge 0$ such that for any $x,y,z\in X$, any triangle with geodesic sides $\Delta xyz$ satisfies the property that any side is in the $\delta$-neighborhood of the union of the other two, then  $(X,d)$  is said to be $\delta$-\emph{hyperbolic}.  

The following result is sometimes referred to as the Morse Lemma in the literature.  
\begin{theorem}[see e.g. \cite{BH}]\label{Morse}
For all $\delta\ge 0$, $\epsilon>0$,  and $\lambda\ge 1$, there is a constant $D$ such that the following holds:
Suppose $Y$ is a $\delta$-hyperbolic metric space.  Then the Hausdorff distance between a geodesic and a $(\lambda,\epsilon)$-quasi-geodesic joining the same pair of endpoints is no more than $D$.  
\end{theorem}

\subsection{Teichm\"uller Space}
Let $S$ be a closed, oriented surfaced of genus $g\ge 2$.  The \emph{Teichm\"uller space} $\teich(S)$ is formed from the set of pairs $(g,X)$ where $g: S\to X$ is a homotopy equivalence and $X$ is a hyperbolic structure on $S$.  Two pairs $(g,X)$ and $(h, Y)$ are said to be equivalent if there is an isometry $\iota: X\to Y$ such that $\iota\circ g\sim h$.  The equivalence class of a pair is denoted $[g,X]$, and the Teichm\"uller space of $S$ is the set of equivalence classes of such pairs, with topology defined by the Teichm\"uller metric.  Briefly, if $\bar g$ is a homotopy inverse for $g$, the Teichm\"uller distance is defined by 
\[d_{\teich(S)}([g,X], [h,Y]) = \inf_{f\sim h\circ \bar g}\frac12\log(K_f)\]
where $K_f$ is the maximum of the pointwise quasiconformal dilatation of the map $f$.  Thus the Teichm\"uller metric measures the difference between the marked conformal structures determined by $(f,X)$ and $(g,Y)$.
Teichm\"uller's Theorems imply that there is a distinguished map $T: X\to Y$, called a Teichm\"uller map, in the homotopy class of $h\circ \bar g$, which achieves the minimal quasiconformal dilatation among all maps homotopic to $h\circ \bar g$.    By abuse of notation, we will often suppress markings and write $X\in\teich(S)$ to refer to the class of a marked hyperbolic structure $g: S\to X$.  We also do not distinguish between pairs and equivalence classes of pairs, but we freely precompose markings with homeomorphisms isotopic to the identity on $S$ to stay within an equivalence class.

\subsection{Ends of hyperbolic $3$-manifolds} \label{ends 1}
Tameness of manifolds with incompressible ends was proved by Bonahon \cite{bon}.  Canary proved that topological tameness implied Thurston's notion of \emph{geometric} tameness \cite{canary:ends}.  For discrete and faithful representations $\rho: \pi_1(S) \to \G$, the quotient $M_\rho$ is  diffeomorphic to the product $S\times \bR$.  Thus $\rho$ determines a homotopy class of maps $[f: S\to M_\rho]$ inducing (the conjugacy class of) $\rho$ at the level of fundamental groups.  $M_\rho$ has two ends, which we think of as $S^+ = S\times \{\infty\}$ and $S^- = S\times \{-\infty\}$. Let $\g = \im \rho$.  The {\it limit set} $\Lambda_\g$  is the set of accumulation points of $\g.x\subset \partial\bH^3$ for some (any) $x\in \partial\bH^3$.  The {\it domain of discontinuity } of $\g$ is $\Omega_{\g} = \hat\bC \setminus \Lambda_{\g}$.
Denote by $\cH(\Lambda_\g)\subset \bH^3$ the convex hull of $\Lambda_\g$.  The \emph{convex core} of $M_\rho$ is $\core(M_\rho) = \cH(\Lambda_\g)/\g$.     

Let $e \in\{+, -\}$.   Say that $S^e$  is {\it geometrically finite} if there is some neighborhood of $S^e$  disjoint from $\core(M_\rho)$.  Call $S^e$ {\it geometrically infinite} otherwise.
By the Ending Lamination Theorem (see Section \ref{section:ELT}) the isometry type of $M_\rho$ is uniquely determined by the surface $S$ together with its {\it end invariants}  $\nu= (\nu(S^-), \nu(S^+))$.  We describe the end invariants $\nu(S^e)$ below.

\subsection{Geometrically finite ends}   \label{ends 2}
 
When $S^e$ is geometrically finite, there is a non-empty component $\Omega_{\g}^e\subset \Omega_{\g}$ on which $\g$ acts freely and properly discontinuously by conformal automorphisms.  This action induces a conformal structure $X=\Omega_{\g}^e/ \g$ on $S^e$ with marking $f: S\to X$ inducing $\rho$ on fundamental groups.  Moreover, the end of $M_\rho$ given by $S_e$ admits a conformal compactification by adjoining $X$ at infinity.  If $S^e$ is geometrically finite, we associate the end invariant $\nu(S^e) = [f,X]\in \teich(S)$. 
 
\subsection{Geometrically infinite ends}\label{infiniteends}
Suppose $S^e$ is geometrically infinite, and let $E_e\cong S\times [0,\infty)$ be a neighborhood of $S^e$. Then $E_e$ is \emph{simply degenerate}.  That is, there is a sequence $\{\gamma_i^*\}$ of homotopically essential, closed geodesics exiting $E_e$.  Each $\gamma_i^*$ is homotopic in $E_e$ to a simple closed curve $\gamma_i\subset S\times\{0\}$.   Moreover, we may find such a sequence such that the length $\ell_{M_\rho}(\gamma_i^*)\le L_0$, where $L_0$ is the Bers constant for $S$.  Equip $S=S\times \{0\}$ with any hyperbolic metric.  Find geodesic representatives $\gamma_i^*\subset S$ with respect to this metric.  Then up to taking subsequences, the projective class of the intersection measures $[\gamma_i^*]\in \cPML(S)$ converges to the projective class of a measured lamination $[\lambda_e]\in \cPML(S)$.   Thurston \cite{Thurston:notes}, Bonahon \cite{bon}, and Canary \cite{canary:ends} show, in various contexts, that the topological support $\lambda_e\subset S$ is minimal, filling, and does not depend on the exiting sub-sequence $\{\gamma_i\}$.  Furthermore, given any two hyperbolic structures on $S$, the spaces of geodesic laminations are canonically homeomorphic.  So $\lambda_e$ also does not depend on our choice of metric on $S$.  This \emph{ending lamination} is the end invariant $\nu(S^e) = \lambda_e$.  Call $\cEL(S)$ the space of minimal, filling laminations.  

\subsection{Pleated surfaces} A \emph{pleated surface} is a map $f: S \to M_\rho$ together with a hyperbolic structure  $X\in \teich(S)$, and a geodesic lamination $\lambda$ on $S$ so that $f$ is length preserving on paths, maps leaves of $\lambda$ to geodesics, and is locally geodesic on the complement of $\lambda$. Pleated surfaces were introduced by Thurston  \cite{Thurston:notes}.  We insist also that  $f$ induces $\rho$ on fundamental groups, \ie it is in the homotopy class of the marking $S\to M_\rho$.  We write $f: X \to M_\rho$ for a pleated surface.  The \emph{pleating locus} of $f$ is denoted by $\pleat(f)$; it is the minimal lamination for which $f$ maps leaves geodesically.

\subsection{The Ending Lamination Theorem} \label{section:ELT}

By Thurston's Double Limit Theorem \cite{Thurston:double}, any pair of end invariants $\nu = (\nu_-, \nu_+) \in (\teich(S)\sqcup \cEL(S))^2 \setminus \Delta(\cEL(S))$ can be realized as the end invariants of a hyperbolic structure on $S\times \bR$.  It was a program of Minsky to establish the following converse, conjectured by Thurston.  We state here a special case of the more general theorem for finitely generated Kleinian groups.
\begin{theorem}[Ending Lamination Theorem]\label{ELT}
Let $S$ be a closed, orientable surface of genus at least $2$ and $\rho_0, \rho_1 : \pi_1(S)\to \G$ be discrete and faithful representations.  There is an orientation preserving isometry $\iota: M_{\rho_0} \to M_{\rho_1}$ such that $\iota \circ f_0 \sim f_1$ if and only if $\nu(M_{\rho_0}) = \nu(M_{\rho_1})$.  Equivalently, there is a $\gamma \in \G$ such that $\rho_0 = \gamma\circ \rho_1 \circ \gamma\inverse$ if and only if $\nu(M_{\rho_0}) = \nu(M_{\rho_1})$. 
\end{theorem}

In the case that $M_{\rho_i}$ have positive injectivity radius, Theorem \ref{ELT} was proved by Minsky  \cite{Minsky:ELTbdd} (see Section \ref{ELTbdd}).  For the general case, Masur and Minsky initiated a detailed study of the geometry of the \emph{curve complex} of $S$ \cite{masur-minsky:complex1} as well as its hierarchical structure \cite{masur-minsky:complex2}.  Given a representation $\rho: \pi_1(S)\to \G$, Minsky extracts the end invariants $\nu = (\nu_-, \nu_+)$ of $M_\rho$ and then uses the hierarchy machinery to build a model manifold $\model M_\nu$ and Lipschitz homotopy equivalence $\model M_\nu \to M_\rho$ \cite{minsky:ELTI}.   Brock, Canary, and Minsky then promote $\model M_\nu \to M_\rho$ to a bi-Lipschitz homotopy equivalence \cite{brock-canary-minsky:ELTII}.  An application of the Sullivan Rigidity Theorem then concludes the proof of the Ending Lamination Theorem in the case of marked Kleinian surface groups.  In this paper, we will use the geometry of the model manifold for bounded geometry ending data from \cite{Minsky:ELTbdd} which is built from a Teichm\"uller geodesic joining $\nu_-$ to $\nu_+$.

\subsection{Teichm\"uller geodesics and models with bounded geometry}\label{ELTbdd}
Suppose $M$ is a hyperbolic structure on $S\times \bR$ with bounded geometry and $f: S \to M$ is a homotopy equivalence.   Fix a basepoint $Y_0\in \teich(S)$ so that for each geometrically infinite end $S^e$ of $M$ and transverse measure $\mu$ supported on $\lambda_e = \nu(S^e)\in \cEL(S)$, there is a unique quadratic differential $\QQ_{\mu}$ that is holomorphic with respect to the conformal structure underlying $Y_0$ and whose vertical foliation is measure equivalent to $\mu$ \cite{Hubbard--Masur:section}.  Since $M$ has bounded geometry, the Teichm\"uller geodesic ray $t\mapsto Y_t\in \teich(S)$, $t\ge 0$ determined by $\QQ_{\mu}$ has bounded geometry \cite[Theorem 1.5]{Rafi:bdd} (see also \cite[Bounded geometry theorem]{Minsky:bounded_geometry}), \ie it projects to a compact subset of the moduli space.  By Masur's Criterion \cite[Theorem 1.1]{Masur:criterion} and because $\lambda_e$ is minimal, the set of transverse measures supported on $\lambda_e$ is equal to $\bR^{>0}\mu$.  This means that $\QQ_{\mu}$ is completely determined by $Y_0$, $\lambda_e$, and the area of the singular falt metric $|\QQ_\mu|$.  

In summary, since $M$ has  bounded geometry, there is a unique projective class of measures supported on $\lambda_e$ and unique quadratic differential $\QQ_e$ holomorphic on $Y_0$ with area $1$ and vertical foliation that is topologically equivalent to $\lambda_e$.  
Minsky proved that the pleated surfaces that can be mapped into an end of $M$ are approximated by a Teichm\"uller geodesic ray in the following sense. 
\begin{theorem}[{\cite[Theorem A]{Minsky:pleated}, \cite[Theorem 5.5]{Minsky:ELTbdd}}] \label{ending ray}
The Teichm\"uller geodesic ray $t\mapsto Y_t\in \teich(S)$ determined by $\QQ_e$ satisfies: for every $t\in [0, \infty)$, there is a pleated surface $f_t: X_t \to M$ homotopic to $f$, such that $d_{\teich(S)}(X_t, Y_t)\le A$, where $A$ depends only on $S$ and $\inj(M) = \epsilon >0$.  
  
\end{theorem}

Assume $M$ is doubly degenerate with end invariants $(\lambda_-, \lambda_+)$; using Theorem \ref{ending ray} we outline the construction of a model metric $ds$ on $S\times \bR$, such that $\model M = (S\times \bR, ds)$ approximates the geometry of $M$.  The nature of this approximation is made precise in Theorem \ref{model map}.  The details of this construction are carried out in \cite[Section 5]{Minsky:ELTbdd}. 
 
Again, since $M$ is $\epsilon$-thick, there are unique projective classes of measured laminations $[\lambda_\pm] \in \cPML(S)$ with supports equal to $\lambda_\pm$.  There exists a hyperbolic structure $Y_0\in \teich(S)$ and a quadratic differential $\QQ_0$ of unit area, holomorphic with respect to the conformal structure underlying $Y_0$, such that $\lambda_+$ is the (hyperbolic straightening of the) vertical foliation $\QQ_0^+$ of $\QQ_0$, and $\lambda_-$ is the horizontal foliation $\QQ_0^-$ of $\QQ_0$.  The conclusion of Theorem \ref{ending ray} holds for $S^-$ and $S^+$, since we can choose $Y_0$ as our basepoint.  
 
 The quadratic differential $\QQ_0$ gives rise to a singular euclidean metric $|\QQ_0|$ on $S$, which we can write infinitesimally as \[ |\QQ_0| = dx^2 +dy^2\] away from the zeros of $\QQ_0$, where $dx$ is the measure induced by $\QQ_0^+$ and $dy$ is the measure induced by $\QQ_0^-$.  Normalizing so that the identity map from $Y_0$ to $Y_t$ is the Teichm\"uller map, the image quadratic differential $\QQ_t$ (holomorphic with respect to $Y_t$) induces a metric $|\QQ_t|$ given by \[|\QQ_t| = e^{2t}dx^2 + e^{-2t}dy^2 \] away from the zeros of $\QQ_0$.  Define a metric $ds$ on $S\times \bR$  by \[ds^2 = e^{2t}dx^2 + e^{-2t}dy^2 +dt^2, \] where $t$ denotes arclength in the Teichm\"uller metric.  The singularities of this metric are exactly $\Sigma\times \bR$, where $\Sigma \subset S\times \{0\}$ is the set of zeros of $\QQ_0$. 
 
 \begin{theorem}[{\cite[Theorem 5.1]{Minsky:ELTbdd}}]\label{model map}
 There is a homotopy equivalence \[ \Psi: \model M = (S\times \bR, ds)\to M\] inducing $f_*$ on fundamental groups that lifts to an $(L, c)$-quasi-isometry $\widetilde{\Psi} :\widetilde{\model M} \to \widetilde M$ of universal covers, where $L$ and $c$ depend only on $\epsilon$ and $S$. 
 The map $\Psi$ satisfies the following properties:
 \begin{enumerate}[(i)]
 \item For each $n\in \bZ$, $\Psi(~\cdot ~, n) = f_n$, where $f_n: X_n\to M$ is the pleated map as in Theorem \ref{ending ray}.
 \item The identity mapping of $S$ lifts of an $(L,c)$-quasi-isometry of universal covers with respect to the singular flat metric $|\QQ_n|$ and the hyperbolic metric $X_n$. 
 \end{enumerate}
 If $M_k$ is the singly degenerate with end invariants $(Y_k, \lambda_+)$, then we have \[\Psi_k : \model M_k = (S\times [k,\infty),ds) \to \core (M_k)\] satisfies the same properties as above.  In addition, $\Psi(~\cdot ~, k)$ maps to $\partial \core(M_k)$. 
 \end{theorem}
 
\section{Singly degenerate classes}\label{singly degenerate classes}

We will use the Geometric Inflexibility Theorem of Brock--Bromberg \cite[Theorem 5.6]{Brock-Bromberg:Inflexibility} that generalizes a result of McMullen \cite[Theorem 2.11]{McMullen:Renormalization}.  Roughly, Geometric Inflexibility says that the deeper one goes into the convex core of a hyperbolic manifold, the harder it is to deform the geometry, there. In this section, we use only the volume preserving and global bi-Lipschitz constant.  In later sections, we use the pointwise, local bi-Lipschitz estimates.  Given a homotopy equivalence $\phi: M_0\to M_1$ of complete hyperbolic manifolds, say that $M_1$ is a $K$-quasiconformal deformation of $M_0$ if $\phi$ lifts to universal covers and continuously extends to an equivariant $K$-quasiconformal homeomorphism $\partial \bH^3\to \partial\bH^3$.  
\begin{theorem}[Geometric Inflexibility, {\cite[Theorem 5.6]{Brock-Bromberg:Inflexibility}}]\label{inflexibility}
Let $M_0$ and $M_1$ be complete hyperbolic structures on a 3-manifold $M$ so that $M_1$ is a $K$-quasiconformal deformation of $M_0$, $\pi_1(M)$ is finitely generated, and $M_0$ has no rank-one cusps.
There is a volume preserving $K^{3/2}$-bi-Lipschitz diffeomorphism \[\Phi : M_0 \to M_1.\]
whose pointwise bi-Lipschitz constant satisfies \[\log \bilip(\Phi,p)\le C_1e^{-C_2d(p,M_0\setminus\core(M_0))}\]
for each $p\in M_0^{\ge \epsilon}$,where $C_1$ and $C_2$ depend only on $K$, $\epsilon$, and $\area(\partial \core(M_0))$.
\end{theorem}
Some remarks about the statement of Theorem \ref{inflexibility} are in order, since the above formulation is rather spread out over the literature.  Namely, the result quoted above uses work of H. M. Reimann \cite{Reimann:volumepreserving} that unifies constructions of Ahlfors \cite{Ahlfors:inflex} and Thurston \cite[Chapter 11]{Thurston:notes} relating a quasi-conformal deformation at infinity to the internal geometry of a hyperbolic structure; see also \cite[Theorem 5.1]{Brock-Bromberg:Inflexibility} for a summary of the main result of Reimann's work that includes the $K^{3/2}$-bi-Lipschitz constant that is absent in the statement of \cite[Theorem 5.6]{Brock-Bromberg:Inflexibility}.   A self contained exposition of Reimann's work can be found in \cite[\S2.4 and Appendices A and B]{McMullen:Renormalization}.     
From a time-dependent quasi-conformal vector field on $\hat\bC = \partial \bH^3$, Reimann  \cite{Reimann:volumepreserving} analyzes a nice extension to $\bH^3$ via a visual averaging procedure that is natural with respect to M\"obius transformations and that be can integrated to obtain a path through hyperbolic metrics.  The behavior of this path of metrics is controlled by the quasi-conformal constant of the initial vector field at infinity.  The visual average of a quasi-conformal vector field is divergence free, and so the flow is volume preserving (c.f. \cite[Appendices A and B]{McMullen:Renormalization}, especially \cite[Theorems B.10 and B.21]{McMullen:Renormalization}).  Thus the map $\Phi$ from Theorem \ref{inflexibility} is volume preserving (this statement is also absent from \cite[Theorem 5.6]{Brock-Bromberg:Inflexibility}).   

We consider only marked Kleinian surface groups with no parabolic cusps.  We now also restrict ourselves to representations with one geometrically infinite end.  Fix a closed, oriented surface $S$ of negative Euler characteristic.  
In light of Theorem \ref{ELT}, we may supply a pair $\nu = (\nu_-, \nu_+)$ of end invariants to obtain a discrete and faithful representation $\rho(\nu): \pi_1(S)\to \g(\nu)\le\G$ and homotopy equivalence $f_\nu: S\to M_{\rho(\nu)} = \bH^3/\g(\nu)$ inducing $\rho(\nu)$ on $\pi_1$, where $M_{\rho(\nu)}$ has end invariants $\nu$.  The $\G$-conjugacy class of $\rho(\nu)$ and the homotopy class of $f_\nu$ are uniquely determined by $\nu$. 
Let $\lambda\in \cEL(S)$ and $X\in \teich(S)$.  Then we have an orientation reversing isometry $\iota : M_{(X,\lambda)}\to M_{(\lambda, X)}$ such that $\iota \circ f_{(X,\lambda)}\sim f_{(\lambda,X)}$.  Without loss of generality, we will work with manifolds whose `$+$' end is geometrically infinite, and whose `$-$' end invariant is geometrically finite.  That is, we will consider manifolds with end invariants $\nu = (X,\lambda)\in \teich(S)\times \cEL(S)$. 

Fix $\lambda\in \cEL$ and take $\nu_0 = (X_0,\lambda)$ and $\nu_1 = (X_1,\lambda)$.  The goal of this section is to prove
\begin{theorem}\label{main}
With notation as above, we have an equality in bounded cohomology
\[\rho(\nu_0)^*\Vol = \rho(\nu_1)^*\Vol ~\in~ \Hbb^3(\pi_1(S);\bR).\]
\end{theorem}

Now that $\lambda$ and $X_i$ are fixed, we abbreviate $M_i = M_{\nu_i}$, $\rho_i = \rho(\nu_i)$, $f_i = f_{\nu_i}$, and $\g_i = \g(\nu_i)$.    We claim that there is a $K$-bi-Lipschitz volume preserving diffeomorphism $\Phi: M_0\to M_1$, where $K^{2/3}$ is dilatation of the Teichm\"uller map $X_0\to X_1$, \ie  $K^{2/3}=\exp(2d_{\teich(S)}(X_0, X_1)) $.  

Indeed, recall that $X_0$ is the quotient of the domain of discontinuity $\Omega_0$ of $\rho_0$ by the action of $\rho_0$.  By Bers' Theorem \cite{Bers:uniformization}, there is a $K^{2/3}$-quasi-conformal homeomorphism $\varphi: \widehat\bC\to\widehat\bC$ so that $\rho' = \varphi\circ \rho_0\circ \varphi\inverse: \pi_1 (S) \to \G$ is discrete, faithful, and such that the quotient of the domain of discontinuity of $\rho'$ by the action of $\rho'$ corresponds to $X_1$.  By Theorem \ref{inflexibility}, we have a $K$-bi-Lipschitz volume preserving diffeomorphism $\Phi: M_0\to M' = \bH^3/\im\rho'$.  By construction, $X_1$ is the `$-$' end invariant of $M'$, and since $\Phi$ maps curves of bounded length exiting the geometrically infinite end of $M_0$ to curves of bounded length exiting an end of $M'$, $\lambda$ is the other end invariant of $M'$.  By Theorem \ref{ELT}, $\rho'$ is conjugate in $\G$ to $\rho_1$.  Then  $M_1 = M'$, and so $\Phi: M_0 \to M_1$ is desired map.    
We  now wish to consider the difference  
\[ \widehat\omega_0 - \Phi^*\widehat\omega_1 \in \Cbb^3(M_0;\bR).\]

We will now construct an explicit  $2$-cochain $C\in \Cb^2(M_0)$ such that \[dC(\sigma) = \int_{\str_0\sigma}\omega_0-\int_{\str_1\Phi_*\sigma} \omega_1=  \widehat\omega_0(\sigma) - \Phi^*\widehat\omega_1(\sigma).\]
To this end, let $\tau: \Delta_2\to M_0$ be continuous.  We define a homotopy $H(\tau): \Delta_2\times I \to M_1$ such that $H_1 =\str_1\Phi_*\tau$ and $H_0 = \Phi_*\str_0\tau$ as follows.  
First, we consider lifts $\tau_0:\Delta_2\to \bH^3$ of $\Phi_*\str_0\tau$ and $\tau_1: \Delta_2 \to \bH^3$ of $\str_1\Phi_*\tau$ to the universal cover $\bH^3\xrightarrow{\pi}M_1$ such that $\tau_0(v_i) = \tau_1(v_i)$ where $v_i$ are the vertices of the 2-simplex $\Delta_2$.  
\begin{remark}
We would like to define our homotopy to be the geodesic homotopy joining $\tau_0(x)$ to $r(\tau_0(x))$, where $r$ is the nearest point projection onto $\im\tau_1$.  However, it is imperative that edges of $\tau_0$ are mapped to edges of $\tau_1$, and it is not guaranteed that $r$ will accomplish this task.  The following construction of $H(\tau)$ fixes this potential problem.
\end{remark}
Let $[i,j] \subset \Delta_2$ be the edge joining vertices $v_i$ and $v_j$.  The image $\tau_1([i,j])$ is a geodesic segment, and so there is a nearest point projection $r_{i,j}: \bH^3\to \im \tau_1([i,j])$, and since $\im\tau_1\subset\bH^3$ is geodesically convex, we have the nearest point projection $r: \bH^3\to \im \tau_1$.  We note here for use later that by convexity of the distance function on $\bH^3$, each of the $r_{i,j}$ and $r$ are $1$-Lipschitz retractions.  

\begin{figure}[h]
\includegraphics[width=5in]{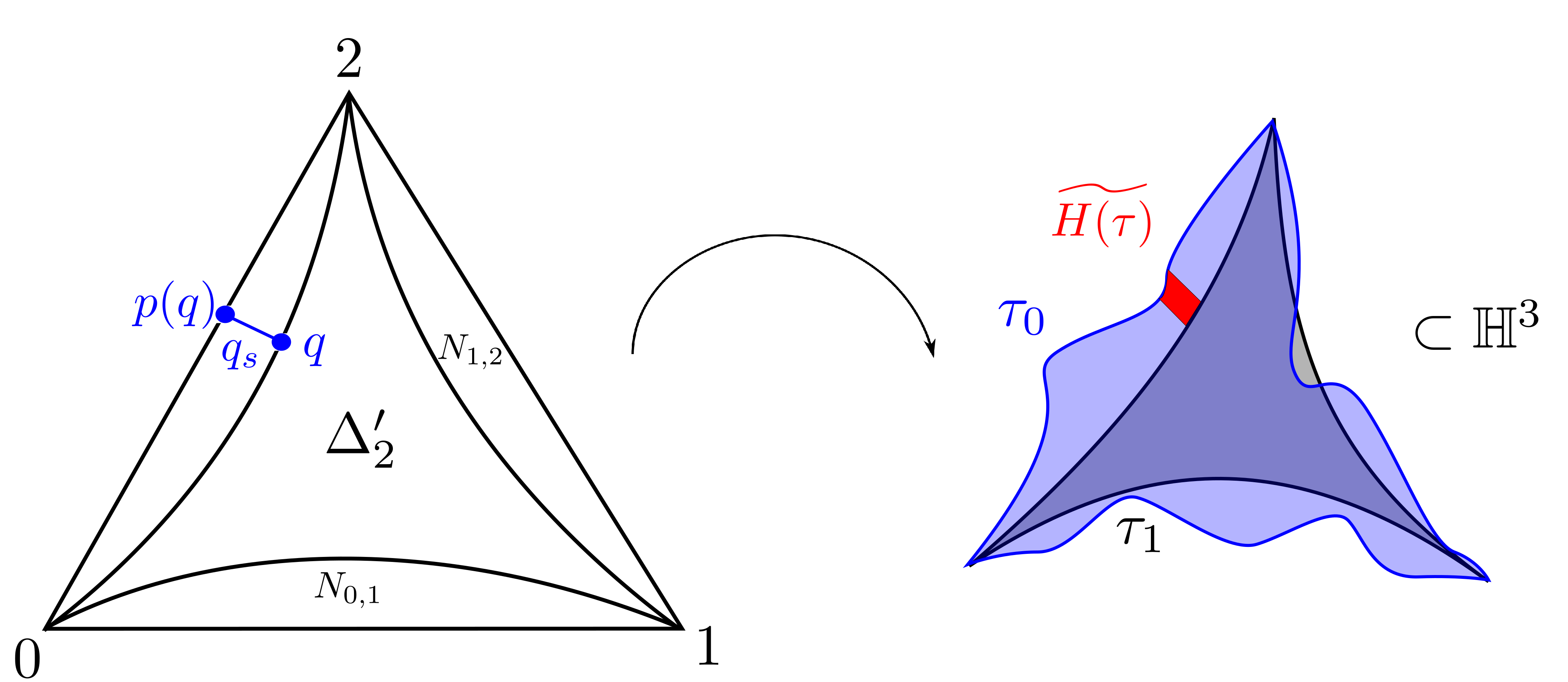}
\caption{Schematic of the definition of $H(\tau)$ in pieces.}
\label{simplex}
\end{figure}

For two points $x,y \in \bH^3$, let $[x,y]: I \to \bH^3$ be the unique geodesic segment, parameterized proportionally to arclength joining $x$ to $y$, \ie $[x,y](0) = x$, $[x,y](1) =y$ and $|[x,y]'(t)| = d(x,y)$.  First we define our map on the edges of $\Delta_2$.  Let $p\in [i,j]$, and define
\[\gamma_t(p) = 
\begin{cases}
[\tau_0(p), r_{i,j}(\tau_0(p))](2t), & 0 \le t\le \frac{1}{2} \\
[r_{i,j}(\tau_0(p)),\tau_1(p)](2t-1), & \frac{1}{2}\le t\le 1.
\end{cases}
\] 
Find regular, convex neighborhoods $N_{i,j}\subset (\Delta_2, d_{\str_0\tau})$  of $[i,j]$ that meet only at the vertices of $\Delta_2$, such that $N_{i,j}\subset \cN_{1/K}( [i,j])$ where distance is measured in the metric induced by $\str_0\tau$.  Let $N = N_{0,1} \cup N_{1,2}\cup N_{2,0}$ and $\Delta_2' = \Delta_2\setminus N$.

For $x\in \Delta_2'$ and $t\in I$, define 
\[\gamma_t(x) = 
\begin{cases}
[\tau_0(x), r(\tau_0(x))](2t), & 0 \le t\le \frac{1}{2} \\
[r(\tau_0(x)),\tau_1(x)](2t-1), & \frac{1}{2}\le t\le 1.
\end{cases}
\]
Finally, let $q \in \partial\cl(N_{i,j})\setminus [i,j]$ and $p(q)\in [i,j]$ be the closest point to $q$ in the $d_{\str_0\tau}$ metric.  
The geodesic segment joining $p$ to $p(q)$ is naturally parameterized proportionally to arclength by convex combinations of $q_s = s\cdot p + (1-s)\cdot(p(q))$, so that $q_1 = q$ and $q_0 = p(q)$.  Define now
\[\gamma_{\frac12}(q_s)= [r_{i,j}(p(q)), r(q)](s),\]
and 
\[\gamma_{t}(q_s)= \begin{cases}
[\tau_0(q_s),\gamma_{\frac12}(q_s)](2t), & 0\le t\le 1/2\\
[\gamma_{\frac12}(q_s), \tau_1(q_s)](2t-1), &  1/2\le t\le 1.
\end{cases}\]
Finally, take
\begin{align*}
H(\tau): \Delta_2\times I &\to M_1\\
(x,t) &\mapsto \pi(\gamma_t(x)).
\end{align*}

Notice that $H(\tau)$ is almost projection of the straight line homotopy between $\tau_0$ and $r\circ \tau_0$ concatenated with a homotopy between $r\circ \tau_0$ and $\tau_1$.  Instead, we constructed $H(\tau)$ to ensure that the edges of $\Delta_2$  do not land in the interior of $\im\tau_1$ at any stage of the homotopy.  We then linearly interpolated between the (potentially different) two maps on the edges over the region $N$.  $H(\tau)$ is piecewise smooth, and by convexity of $\im\tau_1$, the image of the later homotopy is contained entirely within $\tau_1$.   
Define $C\in \Cb^2(M_0)$ by linear extension of the rule
\[ C(\tau) = \int_{\Delta_2\times I }H(\tau)^*\omega_1.\]  

\begin{lemma} \label{coboundary lemma}
Let $\sigma: \Delta_3 \to M_0$. Then 
\[(\widehat\omega_0 - \Phi^*\widehat\omega_1)(\sigma)= dC(\sigma).\]

\end{lemma}
\begin{proof}
Both $\Phi$ and $\Phi\inverse$ are volume preserving, so \[\int_{\str_0 \sigma}\omega_0 = \int_{\Phi_*\str_0 \sigma}{\Phi\inverse}^*\omega_0 = \int_{\Phi_*\str_0 \sigma}\omega_1,\]
so by linearity of the integral 
\[(\widehat\omega_0 - \Phi^*\widehat\omega_1)(\sigma)= \int_{\Phi_*\str_0\sigma - \str_1\Phi_*\sigma}\omega_1.\]

Since $\Hb^3(M_1) = 0$, there is a $2$-form $\eta\in \Omega^2(M_1)$ such that $d\eta = \omega_1$.  The exterior derivative is natural with respect to pullbacks and $\str_i$ are chain maps.  Applying Stokes' Theorem for manifolds with corners, we have
\begin{align}
\int_{\Phi_*\str_0\sigma - \str_1\Phi_*\sigma}\omega_1 &= \int_{\Phi_*\str_0\sigma - \str_1\Phi_*\sigma}d\eta \\ 
&= \int_{\partial(\Phi_*\str_0\sigma - \str_1\Phi_*\sigma)}\eta
=  \int_{(\Phi_*\str_0 - \str_1\Phi_*)(\partial\sigma)}\eta
\end{align}
We write $\partial \sigma = \tau_0-\tau_1+\tau_2-\tau_3$.  By the construction of $H(\tau_i)$, if  $\tau_k|_{[i,j]} = \tau_{k'}|_{[i',j']}$, then the equality $H(\tau_k)|_{[i,j]}(t) = H(\tau_{k'})|_{[i',j']}(t)$ holds for all $t\in I$.  Thus,
\begin{align}
\int_{(\Phi_*\str_0 - \str_1\Phi_*)(\partial\sigma)}\eta &= \sum_{i = 0}^3 (-1)^i\int_{\partial(\Delta_2\times I)}H(\tau_i)^*\eta \\
&= \sum_{i = 0}^3 (-1)^i\int_{\Delta_2\times I}H(\tau_i)^*\omega_1 \\
& = C(\partial\sigma) = dC(\sigma).
\end{align}
This is precisely what we wanted to show.
\end{proof}
The following proposition completes the proof of Theorem \ref{main}.
\begin{prop}\label{coboundary bounded}
There is a $c = c(K)$ such that for each $\tau: \Delta_2\to M_0$, we have \[\left|\int_{\Delta_2\times I } H(\tau)^*\omega_1\right| \le \pi K^2c^3.\]
In other words, $\|C\|_\infty \le \pi K^2c^3$, and so $C\in \Cbb^2(M_0)$.  
\end{prop}
 
\begin{proof}

With $t\in I$ fixed, call $\Delta_t = \Delta\times \{t\}$, and let $\vec n_t(x)$ denote the unit normal vector to the image of the surface $\gamma_t: \Delta_t\to \bH^3$ at $\gamma_t(x)$.  In what follows, $dA_t$ is the Riemannian area form for the pullback of the hyperbolic metric by $\gamma_t$,  $d_t$ is the distance function on $\Delta_t$ induced by this metric, and $d$ is the distance function on $\bH^3$.   We begin by estimating
\begin{align}
\left| \int_{\Delta_2\times I} H(\tau)^*\omega_1 \right| & = \left|\int_0^1  \int_{\Delta_t} \left\langle \frac{\partial H(\tau)}{\partial t} , \vec n_t \right\rangle ~dA_t~dt \right|\\ 
&\le \int_0^{\frac12} \int_{\Delta_t} \left| \left\langle \frac{\partial H(\tau)}{\partial t} , \vec n_t \right\rangle\right| ~dA_t~dt +\int_{\frac12}^1 \int_{\Delta_t} \left| \left\langle \frac{\partial H(\tau)}{\partial t} , \vec n_t \right\rangle\right| ~dA_t~dt \label{integral inequality}
\end{align}
Notice that for $1/2<t\le 1$, we have that $\left\langle \frac{\partial H(\tau)}{\partial t} , \vec n_t \right\rangle = 0$, because the trajectories $\gamma_t(x)$ are traveling orthogonally to $\vec n_t = \vec n_1$ within $\im\tau_1$,  for all $x$. By construction, for $0\le t\le 1/2$, $\left| \frac{\partial H(\tau)}{\partial t}(x,t)\right| = 2|\gamma_t'(x)| = 2d(\gamma_0(x), \gamma_{\frac12}(x))$.  Thus, by Cauchy-Schwartz we have
\[
\left|\left\langle \frac{\partial H(\tau)}{\partial t} , \vec n_t \right\rangle(x)\right| \le 
\begin{cases}
2d(\gamma_0(x), \gamma_{\frac12}(x)), & 0\le t\le 1/2\\
0, & 1/2<t\le 1.
\end{cases}
\]

Combining the above expression with inequality (\ref{integral inequality}), we have 
\begin{equation} \label{recipe}
\left| \int_{\Delta_2\times I} H(\tau)^*\omega_1 \right|\le 2\max_{x\in \Delta_2}\{d(\gamma_0(x), \gamma_{\frac12}(x))\} \int_0^{\frac12}\left( \int_{\Delta_t}dA_t\right)~dt.
\end{equation}

We now show that the distance $d(\gamma_0(x), \gamma_{\frac12}(x))$ is uniformly bounded, independently of $\tau$. 
 $\str_0 \tau|_{[i,j]}$ is a geodesic segment parameterized proportionally to its length $\ell_{i,j}$.  Since $\Phi$ is $K$-bi-Lipschitz, $\tau_0|_{[i,j]}$ is a $(K,0)$-quasigeodesic segment (parameterized proportionally to $\ell_{i,j}$).  Since $\bH^3$ is $\log\sqrt3$-hyperbolic, by the Morse Lemma (Theorem \ref{Morse}) $\tau_0([i,j])$ is no more than $c' = c'( \log\sqrt3,K)$ from $\tau_1([i,j])$, because $\tau_1([i,j])$ is the geodesic segment with the same endpoints as $\tau_0([i,j])$.  
Since hyperbolic triangles are $\log\sqrt3$-thin, given $x\in \Delta_2$, we can find an edge and a point $p \in [i,j]$ such that $x$ is distance at most $\log\sqrt3$ from $p$ in the metric induced by $\str_0\tau$.  If the $\str_0\tau$-segment between $x$ and $p$ passes through $\partial \cl(N)\setminus[i,j]$, call the point of intersection $q$.  Notice that $p = p(q)$ from our definition of $H(\tau)$.  Now $d(\gamma_{\frac12}(p), r(\tau_0(p)))\le c'$, so
\begin{align}\label{traj1}
d(\gamma_{\frac12}(p),\gamma_{\frac12}(x))&\le d(\gamma_{\frac12}(p), r(\tau_0(p)))+ d(r(\tau_0(p)), r(\tau_0(q)))+d(r(\tau_0(q)), \gamma_{\frac12}(x))\\
&\le  c' +1 +K\log\sqrt3
\end{align}
because $d(r(\tau_0(q)), \gamma_{\frac12}(x)) \le K\log \sqrt3$ and since $d_{\str_0\tau}(p,q)\le 1/K$ and $r$ is $1$-Lipschitz, $ d(r(\tau_0(p)), r(\tau_0(q)))\le 1$.  More importantly, 
\begin{align}
 d(\gamma_0(x), \gamma_{\frac12}(x)) & \le d(\gamma_0(x), \gamma_0(p))+d(\gamma_0(p), \gamma_{\frac12}(p)) + d(\gamma_{\frac12}(p),\gamma_{\frac12}(x))\\
 &\le K\log\sqrt3 + c' +(c'+1+K\log\sqrt3) \\ 
 &= K\log 3+2c' +1
 \end{align}
 Since $x$ was arbitrary, we have shown that 
 \begin{equation}\label{bounded tracks}
 \max_{x\in \Delta_2}\{d(\tau_0(x), \gamma_{\frac12}(\tau_0(x)))\} \le c,
 \end{equation} 
 where we define $c:=K\log3+2c'+1$.

We would like to show that the family of identity maps $id_t:(\Delta_2, d_{\str_0\tau}) \to (\Delta_t,d_t)$ are $cK$-Lipschitz on $\Delta_2$.   Let $x,y\in \Delta_2$, and consider the $\str\tau_0$-geodesic segment $[x,y]\subset \Delta_2$.  For $t=0$, we have that $d_0(x, y)\le Kd_{\str_0\tau}(x,y)$, because $\tau_0([x,y])$ is the $K$-Lipschitz image of a geodesic.  Partition the segment $[x,y]$ into intervals $[x=x_0, x_1]\cup[x_1,x_2]\cup ... \cup [x_{n-1}, x_n = y]$.  and approximate $\tau_0([x,y])$ as the concatenation of the geodesic paths $[\gamma_0(x_i),\gamma_0(x_{i+1})]$. By convexity of the hyperbolic metric, and since $r$ is $1$-Lipschitz, if $[x_i, x_{i+1}]\subset \Delta_2'$ we have 
\begin{align}\label{convexity of metric1}
d(\gamma_t(x_i), \gamma_t(x_{i+1})) & \le td(\gamma_0(x_i),\gamma_0(x_{i+1})) + (1-t)d(r(\gamma_0(x_i)),r(\gamma_{0}(x_{i+1})))\\
&\le td(\gamma_0(x_i),\gamma_0(x_{i+1})) + (1-t)d(\gamma_0(x_i),\gamma_{0}(x_{i+1}))\\
& = d(\gamma_0(x_i),\gamma_0(x_{i+1})).\label{convexity of metric3}
\end{align}
Assume now that $[x_i,x_{i+1}]\subset N$.  Then a similar analysis shows that $d(\gamma_t(x_i), \gamma_t(x_{i+1}))\le cd(\gamma_0(x_i),\gamma_0(x_{i+1})$.

Taking a limit over finer partitions, we see that the length $\ell(\gamma_t([x,y]))\le c\ell(\gamma_0([x,y]))$, and so $d_t(x,y)\le cd_0(x,y)$.  Therefore, $d_t(x,y)\le cKd_{\str_0\tau}(x,y)$ and the maps \[id_t: (\Delta_2, d_{\str_0\tau}) \to (\Delta_t,d_t), ~x\mapsto H(\tau)(x,t)\] are $cK$-Lipschitz for all $t\in [0,1]$.  

We can now estimate

\begin{equation} \label{bounded area}
\int_{\Delta_t}dA_t = \int_{id_t(\Delta_2)}dA_t = \int_{\Delta_2}id_t^*dA_t\le \int_{\Delta_2}(cK)^2dA_{\str_0\tau} \le \pi (cK)^2
\end{equation}
because the Lipschitz constant $cK$ bounds the Jacobian of $id_t$ by $(cK)^2$, and the area of a hyperbolic triangle is no more than $\pi$.  Combining estimates (\ref{recipe}), (\ref{bounded tracks}), and (\ref{bounded area}), we obtain

\[
\left| \int_{\Delta_2\times I} H(\tau)^*\omega_1 \right|\le 2c \int_0^{\frac12}(cK)^2\pi~dt = \pi K^2c^3,
\]
which completes the proof of the proposition and Theorem \ref{main}.

\end{proof}

Lemma \ref{coboundary lemma} and Proposition \ref{coboundary bounded} are slightly more general.

\begin{cor}\label{volume preserving bi-lip}
Let $\zeta_0 \in \Omega^3(M_0)$ and $\zeta_1\in \Omega^3(M_1)$, and suppose $F: M_0\to M_1$ is a bi-Lipschitz homeomorphism satisfying $\int_{\str_1 \sigma}F^{-1*}\zeta_0 = \int_{\str_1\sigma}\zeta_1$, for all continuous maps $\sigma: \Delta_3\to M_1$.  If the co-chains $\widehat{\zeta_i}$ are bounded, then \[[F^*\widehat{\zeta_1}] = [\widehat{\zeta_0}] \in \Hb_{\bb}^3(M_0;\bR) \cong \Hb_{\bb}^3(S;\bR).   \]
\end{cor}

Next, we show that singly degenerate bounded fundamental classes can be represented by a class defined on the whole manifold $S\times \bR$, but which has support only in the convex core.  This will come in handy when we prove Theorem \ref{relation 1}.  We remark that the following discussion is essentially standard \cite{Isometric}, but we include arguments here for completeness. 

We will now work with one marked hyperbolic manifold without parabolic cusps $f: S\to M= M_\nu $, and such that  $f_* = \rho(\nu): \pi_1(S) \to \g(\nu)<\G$.  Choose a point $p\in \core(M)$.  We define a new straightening operator  $\str_p:\Cb_\bullet(M)\to \Cb_\bullet(\core(M))$.  Let $\sigma: \Delta_n\to M$ and take any lift of the ordinary straightening $\widetilde{\str\sigma}: \Delta_n\to \bH^3$ to the universal cover $\pi:\bH^3\to M$.  Fix $\tilde{p}\in \pi\inverse(p)$, and let $\mathcal D = \{q\in \bH^3: d(\tilde p , q)\le d(\gamma( \tilde p) , q) \text{ for all } \gamma \in \g(\nu)\}$ be the Dirichlet fundamental polyhedron for $\g(\nu)$ centered at $\tilde p$; delete a face of $\mathcal D$ in each face-pair $(F, \gamma F)$ to obtain a fundamental domain for $\g(\nu)$, which we still call $\mathcal D$.  Then the vertices $v_0, ..., v_n$ of $\widetilde{\str\sigma}$ are labeled  by group elements $v_i = \gamma_i q_i$ where $\rho_\nu(g_i) = \gamma_i\in \g(\nu)$,  $q_i \in \mathcal D$, and $(g_i, q_i) \in \pi_1(S)\times \mathcal D$ is unique.  Define $\str_p \sigma = \pi(\sigma_p(\gamma_0, ..., \gamma_n))$, where $\sigma_p(\gamma_0, ..., \gamma_n)$ is the straightening of any simplex whose ordered vertex set is $(\gamma_0\tilde p, ..., \gamma_n\tilde p)$.  The definition is clearly independent of choices.  Since $p$ is in the convex core of $M$ and all of the edges of $\str_p(\sigma)$ are geodesic loops based at $p$, the edges of $\str_p \sigma$ (hence of all of $\im\str_p\sigma$) are contained in $\core(M)$; moreover, all maps are chain maps and the operator norm $\| \str_p \|\le 1$.  This is just because some simplices in a chain may collapse and cancel after applying $\str_p$.  Define \[\widehat{\omega_p}(\sigma) = \int_{\str_p\sigma}\omega.\]
\begin{remark}
It is not hard to see that the co-chain $\widehat{\omega_p}$ is essentially a topological description of the group co-chain $\rho^*\vol_{\tilde p}$.  
\end{remark}

Now we describe a prism operator on $2$-simplices (the definition extends to $n$-simplices, and so it can be shown that the prism operator defines a chain homotopy between $\str_p$ and $\str$, but we will not need this).  As before, for $x,y\in \bH^3$, $[x,y]: I\to \bH^3$ is the unique geodesic segment joining $x$ to $y$ parameterized proportionally to arclength.  Take the lift $\widetilde{\str\tau}$ with ordered vertex set $(\gamma_0q_0, \gamma_1q_1, \gamma_2q_2)$ where $q_i\in \mathcal D$, and define
\begin{align*}
H_p(\tau): \Delta_2\times I &\to M\\
 (x,t) &\mapsto \pi ([\widetilde{\str\tau}(x),\sigma_p(\gamma_0, \gamma_1, \gamma_2)(x)](t))
\end{align*}
Note that this is just projection of the straight line homotopy between specific lifts of $\str\tau$ and $\str_p\tau$.  
Recall that the prism operator from algebraic topology decomposes the space $\Delta_n\times I$ combinatorially as a union of $(n+1)$-simplices glued along their faces in a consistent way and is used to prove that homotopic maps induce chain homotopic maps at the level of chain complexes.  Apply the prism operator to  $\Delta_2\times I$, to obtain a triangulation comprised of  three tetrahedra inducing a chain $C = C_0+C_1+C_2$.  Then $C_p(\tau) = \str( H_p(\tau)_*C)$ is a straight $3$-chain satisfying $dC_p(\sigma) = C_p(\partial \sigma) = \str_p\sigma - \str\sigma$ for any $\sigma \in \Cb_3(M;\bR)$.  
Moreover,
\begin{equation}\label{prism}
\left|\int_{\str C_p(\tau)}\omega\right| \le 3v_3,
\end{equation}
because $\str C_p(\tau)$ is a sum of $3$ hyperbolic tetrahedra.  We have proved
\begin{lemma}\label{point support}
Let $M$ be a hyperbolic 3-manifold and $p\in \core(M)$.  Then 
 \[[\widehat{\omega}] = [\widehat{\omega_p}]\in \Hb_{\bb}^3(M).\]  
\end{lemma}

\section{Compactly supported bounded classes}\label{compactly supported classes}

To prove Theorem \ref{relation 1}, we will construct a bi-Lipschitz embedding of the core of a singly degenerate manifold into a doubly degenerate one that is volume preserving away from a compact set.  We would like to ignore what happens near that compact set. 
d neighborhood.  The curves still approach the ending lamination and so the/a limiting sequence will exit an end.

If $M$ is a Riemannian manifold and $\omega\in \Omega^k(M)$ is a smooth $k$-form, we define the norm at a point $x\in M$ by \[\|\omega\|(x)  = \sup\{\omega_x(v_1, ..., v_k) : v_i\in T_xM, ~ \|v_i\|\le 1\},\] define $\|\omega\|_\infty = \sup\{\|\omega\|(x): x\in M\}\le \infty$.   
\begin{lemma}\label{bounded primitive}
Let $M$ be a hyperbolic manifold diffeomorphic to $S\times \bR$ with bounded geometry.  Let $\omega \in \Omega_c^3(M)$ be a compactly supported differential $3$-form.  Then there is an $\eta\in \Omega^2(M)$ such that $d\eta = \omega$ and $\|\eta\|_\infty < \infty$.  
\end{lemma}

\begin{proof}
Identify $M$ diffeomorphically with $S\times \bR$.  Since $S\times \bR$ is an open manifold, there is a differential form $\beta\in \Omega^2(S\times \bR)$, such that $d\beta = \omega$.  Without loss of generality, assume that $S\times (-1,1)$ contains the support of $\omega$.  Let $U = S\times (-\infty, -1)$, $W = S\times (-2, 2)$, and $V = S\times (1, \infty)$, and find a partition of unity $\psi = \{\psi_U, \psi_W, \psi_V\}$ subordinate to this cover.  Then $d\beta|_U = 0$ because the support of $\omega$ is contained in the complement of $U$.  The de Rham cohomology $\Hb^2(U) \cong \bR$, and so $[\beta|_U]$ represents the class of $b_U\in \bR$.  For example, since $[S\times \{-10\}]$ is a generator of the degree $2$ homology group, we have \[\langle [\beta|_U], [S\times \{-10\}]\rangle = \int_{S\times\{-10\}} \beta|_U = b_U. \]

Since $M$ has bounded geometry, we can find an embedded geodesic $\ell: \bR\to M$ such that the injectivity radius  $\inj_{\ell(t)}(M)>\epsilon$ for all $t$, and such that $\ell$ exits both ends of $M$.  To see this, choose a basepoint $p\in M$ and sequence $q_n\in M$ at distance $n$ from $p$. Find geodesic segments $[p,q_n]$ from $p$ to $q_n$ of length $n$ exiting an end of $M$.  By Arzela-Ascoli, we can extract a limit $\ell^+$, which is a geodesic ray; also find a ray $\ell^-$ based at $p$ exiting the other end of $M$.  The concatenation of these rays is a quasigeodesic and tracks a geodesic $\ell$ closely; $\ell$ is uniformly thick, because $M$ is.  Moreover, $\ell$ is minimizing in the sense that there is a constant $C>0$ such that $d(p, \ell(t))\ge |t|-C$.  This is because $\ell$ is essentially a limit of minimizing segments making linear progress away from $p$ (see \cite{MahanLecuire} for more about minimizing geodesics).  Since the injectivity radius of $M$ is bounded away from 0, and $\ell$ is minimizing, $\ell$ only comes within $\epsilon$ of itself finitely many times.  Let $m = \max\{n: \{t_1, ... ,t_n\} \text{ are pairwise distinct and } d(\ell(t_1), \ell(t_i))<\epsilon \text{ for all $i$}\}$.  

  Then the Poincar\'e dual of $\im\ell$ is a generator for the de Rham cohomology $\Hb^2(S\times \bR)$; it can be represented by a form $\eta_\ell$ whose support is contained in a tubular $T$ neighborhood of $\ell$.  Indeed, by assumption, the $\epsilon$ neighborhood of $\ell$ is such a tubular neighborhood, foliated by planes meeting $\ell$ orthogonally.  Let $\mathcal O$ be an orthonormal section of the $2$-frame bundle over $T$ which is tangent to each disk $D_\epsilon(\ell(s))$, and let $\psi_\ell$ be a bump function on the disk $D_\epsilon(\ell(s))$ which is  constant on $D_{\frac\epsilon2}(\ell(s))$,  decreases radially to $0$, and has integral $1$.  Then $\eta_\ell=\psi_\ell\mathcal O^*\in \Omega^2(S\times \bR)$ is our desired representative for the poincar\'e dual of $\ell$, and by construction $\|\eta_\ell\|_\infty \le  \frac{m}{\area(D_{\frac\epsilon2})}$. 

Then $b_U\eta_\ell|_U - \beta|_U = d\alpha_U$ for some $\alpha_U\in \Omega^1(U)$.  Construct a new form that interpolates from $b_U\eta_\ell|_U$ to $\beta|_U$ as follows 
\[\gamma_U := \psi_Ub_U\eta_\ell|_U +\psi_W\beta|_U.\]
Then $\gamma_U$ is not closed, but by adding a correction term $d\psi_U\wedge\alpha_U$, one  checks that $d(\gamma_U+d\psi_U\wedge\alpha_U) = 0$.  Moreover, since $d\psi_U\wedge\alpha_U$ has support in $U\cap W$, and $\gamma_U = b_U\eta_\ell|_U$ away from $W$, we have that 
\[\langle [\gamma_U +d\psi_U\wedge\alpha_U], [S\times \{-10\}]\rangle = \int_{S\times\{-10\}} b_U\eta_\ell= b_U.\]
Thus there is an $\alpha_U'\in \Omega^1(U)$ such that $\gamma_U+d\psi_U\wedge\alpha_U - \beta|_U = d\alpha_U'$.  Since $\gamma_U+d\psi_U\wedge\alpha_U$ and $ \beta|_U$ agree near the boundary of $U$, $d\alpha_U'$ is zero there, and so it extends to a $2$-form on all of $S\times \bR$ with support in $U$.  

Repeat these steps on $V$, and define 
\[\eta : = \psi_U(\gamma_U+d\psi_U\wedge \alpha_U) + \psi_W ( \beta - d\alpha_U' - d\alpha_V') + \psi_V(\gamma_V + d\psi_V\wedge \alpha_V).\] A tedious but straightforward calculation shows that $d\eta = \omega = d\beta$. Observe that \[\|\beta|_W -d\alpha'_U|_W - d\alpha'_V|_W+d\psi_U\wedge \alpha_U + d\psi_V\wedge  \alpha_V\|_\infty = c<\infty,\]
because $W$ is precompact, the supports of both $d\psi_U\wedge d \alpha_U$ and $d\psi_V\wedge d \alpha_V$ are compact, and all functions are smooth. Finally we estimate
\[ \|\eta\|_\infty \le m\cdot\max\left\{ \frac{|b_U|}{\area(D_{\frac\epsilon2})}, \frac{|b_V|}{\area(D_{\frac\epsilon2})}\right\} + c,\]
 so $\|\eta\|_\infty<\infty$ and we are finished.

\end{proof}
The previous result was the main technical step for proving the main result from this section, which now follows easily.  The idea is that differential forms define cochains by integration on geodesic simplices.  If the pointwise norm of the differential form is globally bounded with respect to the ambient hyperbolic metric, then this cochain is bounded.  Stokes' Theorem then provides the link between bounded primitives of differential forms and bounded primitives of cochains in negative curvature.  

\begin{prop}\label{compact support}
Let $\rho: \pi_1(S)\to \G$ be Kleinian with bounded geometry and let $f: M_\rho\to \bR$ be compactly supported and smooth.  Then the class $[\widehat{f\omega}] = 0\in \Hb_{\bb}^3(M_\rho;\bR)$.
\end{prop}
\begin{proof}
 By Lemma \ref{bounded primitive}, there is an $\eta\in \Omega^2(M)$ such that $d\eta = f\omega$ and $\|\eta\|_\infty <\infty$.  Since $\str$ is a chain map, by Stokes' Theorem, \[\widehat{\omega}(\sigma)  = \int_{\str\sigma}\omega  = \int_{\str\partial\sigma}\eta = \widehat\eta(\partial\sigma).\]
This means that $\widehat{\omega} = d\widehat{\eta}$. Since the area of a hyperbolic triangle is bounded by $\pi$,   
\[\left|\widehat\eta(\tau)\right| =  \left| \int_{\str\tau }\eta \right| \le  \int_{\str\tau} \|\eta\|~dA\le \pi\|\eta\|_\infty. \]
This shows that $\widehat\eta\in \Cb_{\bb}^2(M)$, and so $[\widehat{\omega}]=0$ in bounded cohomology.  
\end{proof}

\section{Volume preserving limit maps} \label{volume preserving limit maps}
In this section, we consider maps of singly degenerate manifolds into doubly degenerate manifolds that have an end invariant in common.  Since the doubly degenerate manifold is quasi-conformally rigid, we cannot apply the Geometric Inflexibility Theorem to obtain a bi-Lipschitz, volume preserving map.  Instead, we take a limit of such maps.  We will only have control on the bi-Lipschitz constant away from the convex core boundary, so we have to modify our map near it.

Let $\lambda_-, \lambda_+ \in \cEL_{\bb}(S)$, and set $\nu = (\lambda_{\scriptscriptstyle-},\lambda_+)$; then $\inj M_{\nu} = \epsilon>0$.  Then we have a model manifold $\model M_\nu$ and map \[\Psi_\nu : \model M_\nu \to M_{\nu} \] as in Theorem \ref{model map}, a Teichm\"uller geodesic $t\mapsto Y_t $, $t\in \bR$ and pleated surfaces $f_n: X_n \to M_\nu$ for each $n\in \bZ$ as in Theorem \ref{ending ray}.   Since $\Psi_\nu$ is an $(L,c)$-quasi-isometry and $\Psi_\nu(~\cdot ~, n) = f_n$, there is $k\in \bN$ such that 
\begin{equation}\label{nu-separation}
d_{M_\nu}(\im f_{kn}, \im f_{k(n+1)}) > 1
\end{equation} for all $n$.

Let $\nu_{i} = (\lambda_-, Y_{ki})$, $\rho_i = \rho(\nu_i):\pi_1(S)\to \G$, and $M_i = M_{\rho_i}$.  Since $t\mapsto Y_t$ is geodesic, \[d_{\teich(S)} (Y_{ki}, Y_{k(i+1)}) =  \frac12 \log (e^{2k}),\] so by Theorem \ref{inflexibility} we have volume preserving $K= e^{3k}$-bi-Lipschitz diffeomorphsims \[\phi_i : M_i \to M_{i+1} \] and constants $C_1, C_2 >0$ independent of $i$ such that \[\log \bilip (\phi_i, p) \le C_1 e^{-C_2d(p, M_i \setminus \core(M_i))}\] for all $p\in M_i$.   As in Theorem \ref{model map}, there are models $\model M_i$ and maps \[\Psi_i : \model M_i = (S\times (-\infty, ik], ds) \to \core(M_i).\]  Observe that the inclusion $\model M_i \to \model M_\nu$ is  an isometric embedding with respect to the path metric on its image.  Note that $[\rho_i]$ converges to $[\rho_\nu]$ algebraically. 

Let $S_n = S\times \{n\}\subset S\times \bR$.  By Theorem \ref{model map}, for $i\in \bN\cup\{0\}$ and $n\le ki$ the maps $\Psi_i (~\cdot ~ , n): X_n^i \to M_i$ are pleated maps.
We have the following analogue of Inequality (\ref{nu-separation}):
\begin{equation}\label{separation}
d_{M_i}(\Psi_i(S_{kn}), \Psi_i(S_{k(n+1)})) > 1
\end{equation} for all $n< i$.

\begin{prop}\label{volume preserving limit}
 For every $\epsilon' >0$, there is a compact set $\mathcal K \subset \core (M_0)$ and a volume preserving $(1+\epsilon')$-bi-Lipschitz embedding $\Phi_\nu : \core(M_0)\setminus \mathcal K \to  M_\nu$ in the homotopy class determined by $\rho_0$ and $\rho_\nu$.   
\end{prop}

\begin{proof}
The $M_{i}$ will converge geometrically to $M_\nu$, and we will extract a limiting map from the sequence $ \{\Phi_i= \phi_i\circ ...\circ\phi_0|_{\core(M_0)\setminus \mathcal K}\}$ by passing to universal covers.  As long as we can control the bi-Lipschitz constant of $\Phi_i$, we can apply the Arzel\`a--Ascoli Theorem to obtain a convergent subsequence with limit $\widetilde{\Phi_\nu}: \widetilde {\core(M_0)\setminus \mathcal K} \to \bH^3$.  By the algebraic convergence of $[\rho_n]\to [][\rho_\nu]$, $\widetilde{\Phi_\nu}$ is $(\rho_0, \rho_\nu)$-equivariant and hence descends to a map $\Phi_\nu: \core(M_0)\setminus \mathcal K \to M_\nu$ of quotients. The limiting map will be volume preserving away from $\mathcal K$ and $(1+\epsilon')$-bi-Lipschitz, where $\mathcal K$  depends on both $\epsilon'$ and $K$. 

Now we choose $\mathcal K$.    For $N \in \bN$, $\Psi_0(~\cdot~, -N)$ is an immersed homotopy equivalence.  By minimal surface theory \cite{FHS:embedded}, for any $\epsilon''>0$, there is an embedding $g_{-N}: S_{-N} \to \cN_{\epsilon''}(\Psi_0(S_{-N}))$ that is homotopic to $\Psi_0(~\cdot~, -N)$.  By the homeomorpshism theorem of Waldhausen \cite{Waldhausen:homeo}, $\partial \core(M_0)$ and $\im g_{-N}$ bound an embedded submanifold $\mathcal K_N$ homeomorphic to $S\times I$.  

 By Inequality (\ref{separation}), since $\partial \core(M_0) =\Psi_0(S_0)$, 
 \[d_{M_0}(\Psi_0(S_0), \Psi_0(S_{-N})) > \frac Nk.\]   
 
 Choose $N$ large enough so that, using the above estimate and Theorem \ref{inflexibility},
 \[\bilip (\phi_0, \core(M_0)\setminus \mathcal K_{N}) \le (1+\epsilon')^{1-e^{-C_2}}< (1+\epsilon'),\] and take $\mathcal K = \mathcal K_{N+1}$.    

Now we show, by induction, that for $n\ge 0$, 
 \begin{equation}\label{inductive hyp}
 \log\bilip(\phi_n, \Phi_{n-1} (\Psi_0(S_{-N})))\le (1-e^{-C_2})\log(1+\epsilon') e^{-C_2n}.
 \end{equation} 
 From this it will follow that for all $n\ge 0 $,  
 \begin{align}\label{global bound}
 \log\bilip(\Phi_n, \Psi_0(S_{-N}))&\le (1-e^{-C_2})\log(1+\epsilon') \sum_{j = 0}^ne^{-jC_2}\\ 
& < \log(1+\epsilon')\frac{1-e^{-C_2}}{1-e^{-C_2}}=\log(1+\epsilon').
 \end{align}
 The base case was covered, above.  So assume that (\ref{inductive hyp}) holds for all $m<n$, so that $\bilip(\Phi_{n-1}, \Psi_0(S_{-N}))< (1+\epsilon')$.
 It suffices to show that there is some constant $e\ge 0$ such that \[d_{M_n}(\Phi_{n-1}(\Psi_0(S_{-N})), \partial \core(M_n))\ge n+\frac Nk - e.\]
To this end, we would like to show that $d^H_{M_n}(\Phi_{n-1}(\Psi_0(S_{-N})), \Psi_n(S_{-N}))<e$ for some $e$, where $d^H_X$ denotes the Hausdorff distance on closed subsets of a metric space $X$.    
We use the following well known fact, which follows from the observation that the projection onto a geodesic in $\bH^3$ contracts by a factor of at least $\cosh R$, where $R$ is the distance to the geodesic.
\begin{lemma}[{\cite[Lemma 4.3]{Minsky:ELTbdd}}]\label{close curves}
Let $M$ be a hyperbolic manifold, $\alpha: S^1 \to M$ a homotopically non-tivial closed curve, and $\alpha^*\subset M$ its geodesic representative.  Then $d^H_M(\alpha, \alpha^*)\le R$, where $R = \cosh\inverse \left(\frac{\ell_M(\alpha)}{\ell_M(\alpha^*)}\right) +\frac12 \ell_M(\alpha)$.   
\end{lemma}

Let $\alpha\subset S_{-N}$ be a simple closed $X_{-N}^0$-geodesic whose length is at most $L_0$---the Bers constant for $S$.   By property (ii) in Theorem \ref{model map}, the $(S, |\QQ_0|)$-length of $\alpha$ is at most $LL_0+c$, where $L$ and $c$ depend only on $S$ and $\epsilon$.  A second application of property (ii) in Theorem \ref{model map} shows that the $X_{-N}^n$-length of $\alpha$ is at most $L^2L_0+Lc+c$.  Since $\bilip(\Phi_{n-1}, \Psi_0(S_{-N}))<(1+\epsilon')$, $\ell_{M_n}( \Phi_{n-1}(\Psi_0(\alpha)))\le (1+\epsilon')L_0$.  The $M_n$-length of the geodesic representative of $\alpha$ is at least $2\epsilon$, so applying Lemma \ref{close curves}, there is a constant $L' = L'( S, \epsilon, \epsilon')$ such that 
$d_{M_n}^H(\Phi_{n-1}(\Psi_0(\alpha)), \Psi_n(\alpha))\le 2L'$.  There is a universal bound $D= D(\epsilon)$ on the diameter of an $\epsilon$-thick hyperbolic surface.  Thus \begin{equation}\label{surface distance}
d_{M_n}^H( \Psi_n(S_{-N}),\Phi_{n-1}(\Psi_0(S_{-N})))\le D+(1+\epsilon')D+2L' = e.\end{equation}

Recall that $\partial \core(M_n) = \Psi_n(S_{kn})$.  We use Inequalities (\ref{surface distance}) and (\ref{separation}) together with the triangle inequality to estimate
\begin{align*}
d_{M_n}(\partial \core(M_n), \Phi_{n-1}(\Psi_0(S_{-N}))) & \ge d_{M_n}(\Psi_n(S_{kn}), \Psi_n(S_{-N})) - d^H_{M_n}( \Psi_n(S_{-N}),\Phi_{n-1}(\Psi_0(S_{-N}))) \\ 
&\ge \frac Nk+n -e.
\end{align*}

We (retroactively) assume that $N/k-e$ was large enough so that $C_1e^{-C_2(\frac Nk-e)}<(1-e^{-C_2})\log(1+\epsilon')$.  Then $\log\bilip(\phi_n, \Phi_{n-1}( \Psi(S_{-k})))\le (1-e^{-C_2})\log(1+\epsilon')e^{-C_2n}$, which completes the proof of (\ref{inductive hyp}) and also (\ref{global bound}).  
\end{proof}

\begin{remark} Later on, we will only use the existence of a bi-Lipschitz embedding $\Phi_-: \core(M_0) \to  M_\nu$ that is volume preserving away from a compact set.  We can therefore take $\epsilon = 1$  (this is an arbitrary choice) in the statement of  Proposition \ref{volume preserving limit}, and extend $\Phi_\nu$ to $\mathcal K$ by \emph{any} bi-Lipschitz homeomorphism onto some subset $\mathcal K'\subset M_\nu\setminus \im \Phi_\nu$ to obtain $\Phi_-$.  The overall bi-Lipschitz constant will be $B = \max\{2, \bilip(\Phi_-|_{\mathcal K})\}$.  

\end{remark}

\section{Ladders}\label{ladders}

We will want to understand the relationship between based geodesics loops in a doubly degenerate manifold $M$ with bounded geometry and with end invariants $\nu = (\lambda_{\scriptscriptstyle-}, \lambda_+)$ and the same homotopy class of based geodesic loop in a singly degenerate manifold $N$ with end invariants $(Y_0,\lambda_+)$.  Here, $Y_0$ is a point on the Teichm\"uller geodesic between $\lambda_{\scriptscriptstyle-}$ and $\lambda_+$, as in Section \ref{ELTbdd}.  The model manifolds have the form $\model N = (S\times \bRplus, d_{\model N})$ and $\model M = (S\times \bR, d_{\model M})$, where the infinitesimal formulation of the metrics $d_{\model M}$ and $d_{\model N}$ are both of the form $ds^2 = |\QQ_t| +dt^2$, as explained in Section \ref{ELTbdd} .  Then the inclusion $\widetilde{\model N} \hookrightarrow \widetilde{\model M}$ of universal covers is isometric with the path metric on the image, but  $\widetilde{\model N} \subset \widetilde{\model M}$ is very far from being even qausi-convex .  The model manifold $\model N $ has universal cover that is equivariantly quasi-isometric to $\widetilde{\core(N)}$.   

let $\Sigma\subset S_0$ be the set of zeros of $\QQ_0$. We pick a point $p \in \Sigma \subset \model N \subset \model M$, and choose $\tilde p$ a lift of $p$ under universal covering projections.  Let $\iota_t^*d$ be singular Euclidean metric on $\widetilde{S_t} = \widetilde{S\times\{t\}}$ defined by  $|\widetilde{\QQ_t}|$, so that the inclusions  $\iota_t:\widetilde {S_t}\to \widetilde{\model M}$ as level surfaces are isometries with respect to the path metric on the image.  We withhold the right to  abbreviate a point $(x,t)\in \widetilde{S_t}$ as $x_t$.  Note that the (identity) map $\widetilde{S_s}\to \widetilde{S_t}$, $x_s\mapsto x_t$ is a lift of the Teichm\"uller map between the conformal structures underlying $\QQ_s$ and $\QQ_t$, by construction (see Section \ref{ELTbdd}).  That is, $x_s\mapsto x_t$ is affine, $e^{|s-t|}$-bi-Lipschitz, and it takes $\iota_s^*d$-geodesics to $\iota_t^*d$-geodesics.  The inclusions $(\widetilde{S_t}, \iota_t^*d)\to (\widetilde{\model M}, d_{\model M})$ are uniformly proper for all $t$ in the following sense. 
\begin{lemma}\label{uniformly proper} For all $t\in \bR$ and  $D\ge 0$, if $d_{\model M}(x_t,y_t)= D$, then $\iota_t^*d(x_t, y_t)\le D e^{D/2}$.
\end{lemma}
\begin{proof}Let $\gamma: [0,D]\to \widetilde{\model M}$ be a $d_{\model M}$-geodesic joining $x_t$ and $y_t$.  Then $\im\gamma\subset \widetilde{S} \times [t-D/2, t+D/2]$.  Endow $\widetilde{S} \times [t-D/2, t+D/2]\subset\widetilde{\model M}$ with its path metric; the projection  $p_t:  \widetilde{S} \times [t-D/2, t+D/2] \to S_t$, mapping $x_s\mapsto x_t$ is then a $e^{D/2}$-Lipschitz retraction that commutes with the bi-Lipschitz Teichm\"uller map; they are both identity on the first coordinate.  So the length of $p_t(\gamma)\subset \widetilde{S_t}$ is at most $e^{D/2}D$, which bounds the distance $\iota_t^*d(x_t,y_t)$.  
\end{proof}

Suppose $\gamma: I \to S_0$ is a closed $d_0$-geodesic based at $p$.  Find the lift $\widetilde \gamma : I \to \widetilde S_0$ based at $\tilde p$.  We will conflate the map $\widetilde \gamma$ with its image $\widetilde \gamma \subset \widetilde {\model N} \subset \widetilde{\model M}$. While $\model M$ and $\model N$ do not have negative sectional curvatures, they are $\delta$-hyperbolic, where $\delta$ depends on both $S$ and $\epsilon$ as in Theorem \ref{model map}; the Teichm\"uller geodesic $Y_t\subset \teich^{\ge \epsilon}(S)$.  Straighten $\widetilde{\gamma}$ with respect to the two metrics $\gamma^*\subset \widetilde{\model M}$ and $\gamma_+^*\subset \widetilde{\model N}$ rel endpoints.  We would like to make precise the notion that $\gamma^*$ and $\gamma_+^*$ fellow travel in $\widetilde{\model N}$, and while $\gamma^*$ spends time in $\widetilde{\model M}\setminus\widetilde{\model N}$, $\gamma_+^*$ spends time near $\widetilde{\partial \model N}$.   The main result from this section is

\begin{theorem}\label{almost bounded distance}
There exists a constant $D_1>0$ depending only on $\epsilon$ and $S$ such that the following holds.
Let $\gamma: I \to \model N$ be a loop based at $p$. With notation as above, let $\Pi: \widetilde {\model M} \to \im \gamma^*$ be a nearest point projection.  

If $d(\gamma_+^*(t), \Pi(\gamma_+^*(t)))>D_1$, then $\Pi(\gamma_+^*(t))\in \mathcal N _{D_1}(\widetilde{\model M}\setminus \widetilde{\model{N}})$.  In other words, $\gamma_+^*\subset \cN_{D_1}(\gamma^* \cup \partial \widetilde{\model N})$ and $\gamma^*\cap \widetilde{\model N} \subset \cN_{D_1}(\gamma_+^*)$. 
\end{theorem}

We prove Theorem \ref{almost bounded distance} at the end of the section after first describing a family of  quasi-geodesics in $\widetilde{\model N}\subset \widetilde{\model M}$. With $\gamma$ and $p$ as above, let $\ladder(\gamma) = \widetilde\gamma \times \bR$ and $\ladder_+(\gamma) = \widetilde\gamma \times \bRplus$.  The spaces $\ladder(\gamma)\subset \widetilde{\model M}$ and $\ladder_+(\gamma)\subset \widetilde{\model N}$ can be thought of as the union of the lifts of based $\iota_t^*d$-geodesic segments as $t$ ranges over $\bR$ or $\bRplus$.  Each of the spaces $\ladder(\gamma)\subset \widetilde{\model M}$ and  $\ladder_+(\gamma)\subset \widetilde{\model N}$ inherits a path metric that we call $d_\ladder$.  We call the spaces $(\ladder(\gamma), d_\ladder)$ and $(\ladder_+(\gamma), d_\ladder)$ \emph{ladders} after \cite{Mahan:CT-bdd}, \cite{Mahan:CT-trees}.  The point of the ladder construction is that we will be able to reduce the problem of understanding based geodesics in $\model M$ or $\model N$ to understanding the way based geodesics behave in $2$-dimensional quasi-convex subsets.
 
 Ladders also inherit metrics from the ambient spaces $\widetilde{\model M}$ and $\widetilde{\model N}$.  We now show that the inclusion of a ladder with its path metric is undistorted with respect to the ambient model metric, and so ladders are quasi-convex.  The strategy of the proof is standard and can be found in \cite{Mahan:CT-bdd}.  We include details, because we prefer to continue to work with the infinitesimal formulation of the metric as opposed to discretizing our space.  
\begin{prop}\label{quasi-convex}
There are constants $L_0$ and $L_1$ depending on $S$ and $\epsilon$, such that the inclusion $(\ladder(\gamma), d_\ladder)\hookrightarrow (\widetilde{\model M}, d_{\widetilde{\model M}})$ is an $(L_0,L_0)$-quasi-isometric embedding.  Consequently, $\ladder(\gamma)\subset \widetilde{\model M}$ is $L_1$-quasi-convex.  Analogous statements hold for $\ladder_+(\gamma)\subset \widetilde{\model N}$.  
\end{prop}

\begin{proof}
We will construct an $(L_0, L_0)$-coarse-Lipschitz retraction $\Pi^{\gamma} : \widetilde{\model M} \to \ladder(\gamma)$.  Once we have done so, we see that for each $x,y\in \ladder$, we have 
\[d_\ladder(x,y) = d_\ladder(\Pi^\gamma(x), \Pi^\gamma(y)) \le L_0 d_{\widetilde{\model M}}(x,y) +L_0.\] 
By definition of the path metric, for all $x,y\in \ladder$, we have $d_{\widetilde{\model M}}(x,y)\le d_\ladder(x,y)$.  Combining inequalities, we see that \[ d_{\widetilde{\model M}}(x,y)\le d_\ladder(x,y)\le L_0 d_{\widetilde{\model M}}(x,y) +L_0,\] which is to say that $\ladder\hookrightarrow \widetilde{\model M}$ is an $(L_0, L_0)$-quasi-isometric embedding.  Since $(\widetilde{\model M},d_{\widetilde{\model M}})$ is a $\delta$-hyperbolic metric space, $\ladder$ is $\delta' = \delta'(\delta, L_0)$ hyperboic, and by the Morse Lemma (Theorem \ref{Morse}), $\ladder$ is $L_1 = L_1(\delta, L_0)$-quasi-convex. 

Now we construct $\Pi^\gamma$.
Let $\pi^\gamma_t: \widetilde{S} \to \gamma$ be the nearest point projections with respect to the metric $\iota_t^*d$; the spaces $(\widetilde{S}, \iota_t^*d)$ are CAT(0) and $\gamma_t$ is $\iota_t^*d$-geodesically convex, so the projections $\pi^\gamma_t$ are $1$-Lipschitz retractions \cite{BH}.
The global retraction 
\begin{align*}
\Pi^{\gamma} : \widetilde{\model M} & \to \ladder(\gamma) \\ (x,t) & \mapsto (\pi^\gamma_t(x),t)
\end{align*}
leaves $dt^2$ invariant.  
We now need to understand how much the projections $\pi^\gamma_t$ change as $t$ changes, because unlike in Lemma \ref{uniformly proper}, the projections $\pi_t^\gamma$ do not commute with $x_s = (x,s) \mapsto (x,t) = x_t$. 

Assume $d_{\widetilde{\model M}}(x_s, y_t)\le 1$.  We will find $L_0\ge 1$ such that $d_\ladder(\Pi^\gamma(x_s), \Pi^\gamma(y_t))\le L_0$, and the lemma will follow from this by a  repeated application of the triangle inequality.  In what follows we abbreviate $\pi^\gamma_t$ to $\pi_t$.
\begin{align}
\label{line 1} d_{\ladder}(\Pi^\gamma(x_s), \Pi^\gamma(y_t)) &= d_{\ladder}((\pi_s(x),s), (\pi_t(y),t))\\  \label{line 2}
& \le d_{\ladder}((\pi_s(x),s), (\pi_s(x),t))+d_{\ladder}((\pi_s(x),t), (\pi_t(x),t))+d_{\ladder}((\pi_t(x),t), (\pi_t(y),t)) \\ \label{line 3}
&\le |s-t| + \iota_t^*d((\pi_s(x),t), (\pi_t(x),t)) +\iota_t^*d((\pi_t(x),t), (\pi_t(y),t))\\ \label{line 4}
& \le 1 + C + \iota_t^*d(x_t, y_t) 
\end{align}
Equality in (\ref{line 1}) is the definition of $\Pi^\gamma$, and (\ref{line 2}) is the triangle inequality.  In (\ref{line 3}), we use the fact that segments of the form $(x, I)$ for any connected interval $I\subset \bR$ are convex and have length $|I|$, and the fact that $d_\ladder \le \iota_t^*d$ restricted to $S_t\cap \ladder$.   The Teichm\"uller map $S_s\to S_t$ is $e^{|s-t|}$-bi-Lipschitz, hence in particular an $(e,0)$-quasi-isometry, because $|s-t| \le d_{\model M} (x_s, y_t)\le 1$.   Since ``quasi-isometries almost commute with nearest point projections'' in a $\delta$-hyperbolic metric space, e.g. \cite[Lemma 2.5]{Mahan:CT-bdd} or \cite[Proposition 11.107]{Drutu-Kapovich:book},  $\iota_t^*d((\pi_s(x),t), (\pi_t(x),t))\le C = C(\delta, e)$, which is the constant appearing in (\ref{line 4}).  We have also used that $\pi_t: S \to \gamma$ is $1$-Lipschitz in (\ref{line 4}). 

 We assumed that $d_{\widetilde{\model M}}(x_s, y_t)\le 1$ so \[d_{\widetilde{\model M}}(y_t, x_t)\le d_{\widetilde{\model M}}(y_t, x_s) + d_{\widetilde{\model M}}(x_s, x_t) \le 1+|s-t| \le 2.\]  Using Lemma \ref{uniformly proper}, we obtain $\iota_t^*d(x_t, y_t) \le 2e^{2/2} = 2e$.   
Taking $L_0 = 1+C+2e$ completes the proof that $\Pi^\gamma$ is an $(L_0, L_0)$-coarse-Lipschitz retraction.  The proof of the corresponding statements for  $\ladder_+(\gamma)\subset \model N$ are analogous.   
\end{proof}

Let $\alpha: [0, a] \to \widetilde{S_0}$ be a \emph{saddle connection}, \ie $\alpha(0), \alpha(a)\in \Sigma$, $\alpha(t)\not\in\Sigma$ for all $t\in (0, a)$, and $\alpha$ is a $d_0$-geodesic segment parameterized by arclength (so $a$ is the $d_0$ length of $\alpha$). Then $\alpha$ is a Euclidean segment, and its \emph{slope} $s(\alpha)$ is the ratio of the transverse measures of $\alpha$ in the vertical and horizontal directions when those are non-zero; in other words, $\frac{\int_\alpha |dy|}{\int_\alpha |dx|} = |s(\alpha)|$. If $\int_\alpha |dy| = 0$, then $\alpha$ is  an \emph{expanding saddle connection}.  If ${\int_\alpha |dx|} = 0$, then $\alpha$ is a \emph{contracting saddle connection}.  When $\alpha$ is neither an expanding nor a contracting saddle connection, the pullback of the model metric $ds^2$ by inclusion of $\ladder(\alpha)\hookrightarrow \widetilde{\model M}$ has the form \[e^{2t}dx^2 +e^{-2t}s(\alpha)^2dx^2 +dt^2,\] and the value of $t$ where $t\mapsto e^{2t} +e^{-2t}s(\alpha)^2$ is minimized is given by $b(\alpha): = \frac12\log(|s(\alpha)|)$.  Call $b(\alpha)$ the \emph{balance time} of $\alpha$; thus we see that the  the function $t \mapsto \ell_t(\alpha)$ is convex, and in fact takes the form $t\mapsto \ell_{b(\alpha)}(\alpha)\sqrt{\cosh2(t-b(\alpha))}$; $\ell_{b(\alpha)}(\alpha)$ is the unique minimum value.

For any non-expanding and non-contracting saddle connection $\alpha$ whose length is at most $1$ at its balance time, define two numbers $b_{\pm}(\alpha)\in \bR$ by the rule $\ell_{b_{\pm}(\alpha)}(\alpha) = 1$; the two values of $b_{\pm}$ are given explicitly by the function in the previous paragraph, and are defined such that $b_-\le b \le b_+$.  If $\alpha$ is expanding, the length of $\alpha$ in level surfaces is infimized at $b(\alpha) : = -\infty$, while if $\alpha$ contracting, it's length in level surfaces is infimized at $b(\alpha) :=+\infty$. Note however that if $\alpha$ is expanding, then $\ell_{\log a }(\alpha) =1$; set $b_+(\alpha) = \log a$ and $b_-(\alpha) = b(\alpha) = -\infty$.  If $\alpha$ is contracting, then $\ell_{-\log a}(\alpha) =1$; set $b_-(\alpha) = -\log a$ and $b_+(\alpha) = b(\alpha) = +\infty$.   If $\alpha$ is not short at its balance time, set $b_\pm(\alpha) = b(\alpha)$.  See Figure \ref{VH-crossings} for a schematic of the definitions of $b_\pm(\alpha)$ and $b(\alpha)$.  

Say that $\beta : [0, a] \to \ladder(\gamma)$ is a \emph{horizontal segment} if $\im \beta \subset S_t$ for some $t\in \bR$.  Say that $\beta: [0,a]\to \ladder(\gamma)$ is a \emph{vertical segment} if $\im \beta \subset \{x\} \times \bR$ for some $x\in \widetilde{S_0}$.  Say that $\gamma$ is a \emph{VH-path} if it decomposes as a concatenation of subpaths $\gamma = \beta_1\cdot \alpha_1\cdot  ... \cdot \alpha_n\cdot \beta_{n+1}$ that alternate between vertical and horizontal segments, beginning and ending in vertical segments (that may have length $0$).  
\begin{notation} If $\alpha\subset \widetilde{S_0}$ is a saddle connection and $h\in \bR$, then by abuse of notation $\alpha_h$ denotes both the geodesic map $[0, \ell_h(\alpha)]\to \widetilde{S_h}$, $t\mapsto \alpha(t)\times \{h\}$ or $t\mapsto \alpha(\ell_h(\alpha) -t)\times \{h\}$ and the image of that map $\alpha\times \{h\}$.  If $I$ is an index set, $\alpha_i\subset \widetilde{S_0}$ will refer to the $i$th saddle connection in a set $\{\alpha_i\}_{i\in I}$.
\end{notation}

\begin{definition}\label{VH-geodesics}
Let $\alpha\subset \widetilde{S_0}$ be a saddle connection, and let $x_s, y_t$ be points on the two different components of $\partial \ladder(\alpha)$.  Say that a VH-path $\gamma = \beta_r\cdot \alpha_h\cdot \beta_l$ joining $x_s$ and $y_t$ is a \emph{VH-crossing}, if $\beta_r$and $\beta_l$ are geodesics parameterized by arclength, $\alpha_h$ is a $d_h$-geodesic segment, and
\begin{enumerate}[(i)]
\item If $\ell_{b(\alpha)}(\alpha)\ge 1$, or equivalently, $b(\alpha) = b_\pm(\alpha)$, then $h = b(\alpha)$.  
\item If $b_-(\alpha)<b(\alpha)<b_+(\alpha)$, then 
\begin{enumerate}
\item if $s\ge b_+(\alpha)$, then $h = b_+(\alpha)$, and if $s\le b_-(\alpha)$, then $h= b_-(\alpha)$. \label{3a}
\item if $b_-(\alpha)\le s\le b_+(\alpha)$, then $h = s$. \label{3b}
\end{enumerate}
\end{enumerate}
\end{definition}
We aim to prove that VH-crossings are efficient, which is partly the content of Lemma \ref{Hdist}, Lemma \ref{Hdistplus1}, and Lemma \ref{Hdistplus2}, below.  We now explain how to think about ladders as convex subsets of the hyperbolic plane bound by bi-infinite geodesics; VH-crossings mimic the behavior of geodesics in hyperbolic geometry that cross this convex set.
\begin{figure}[h]
\def\svgwidth{\columnwidth}
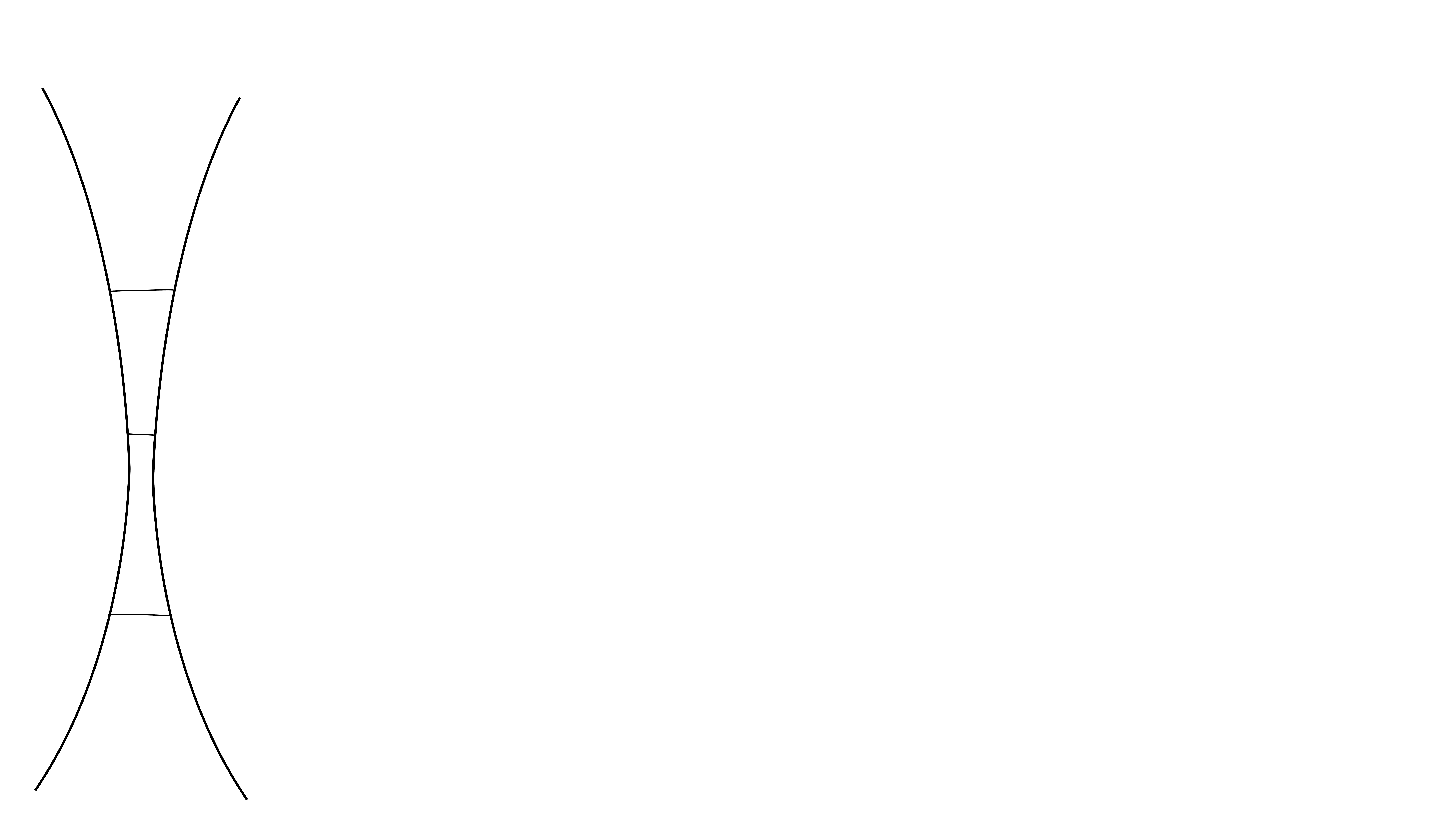
\caption{On the left, $\ell_{b(\alpha)}(\alpha)<1$.  The VH-crossings between some points, $a$, ..., $f\in \partial \ladder(\alpha)$ are depicted in red.  On the right, $
\beta$ is a contracting saddle connection, and $i_\beta$ maps $\ladder(\beta)$ to $\mathbb S(0)\subset \bU$.  }
\label{VH-crossings}
\end{figure}
\begin{figure}[h]
\def\svgwidth{.8\columnwidth}
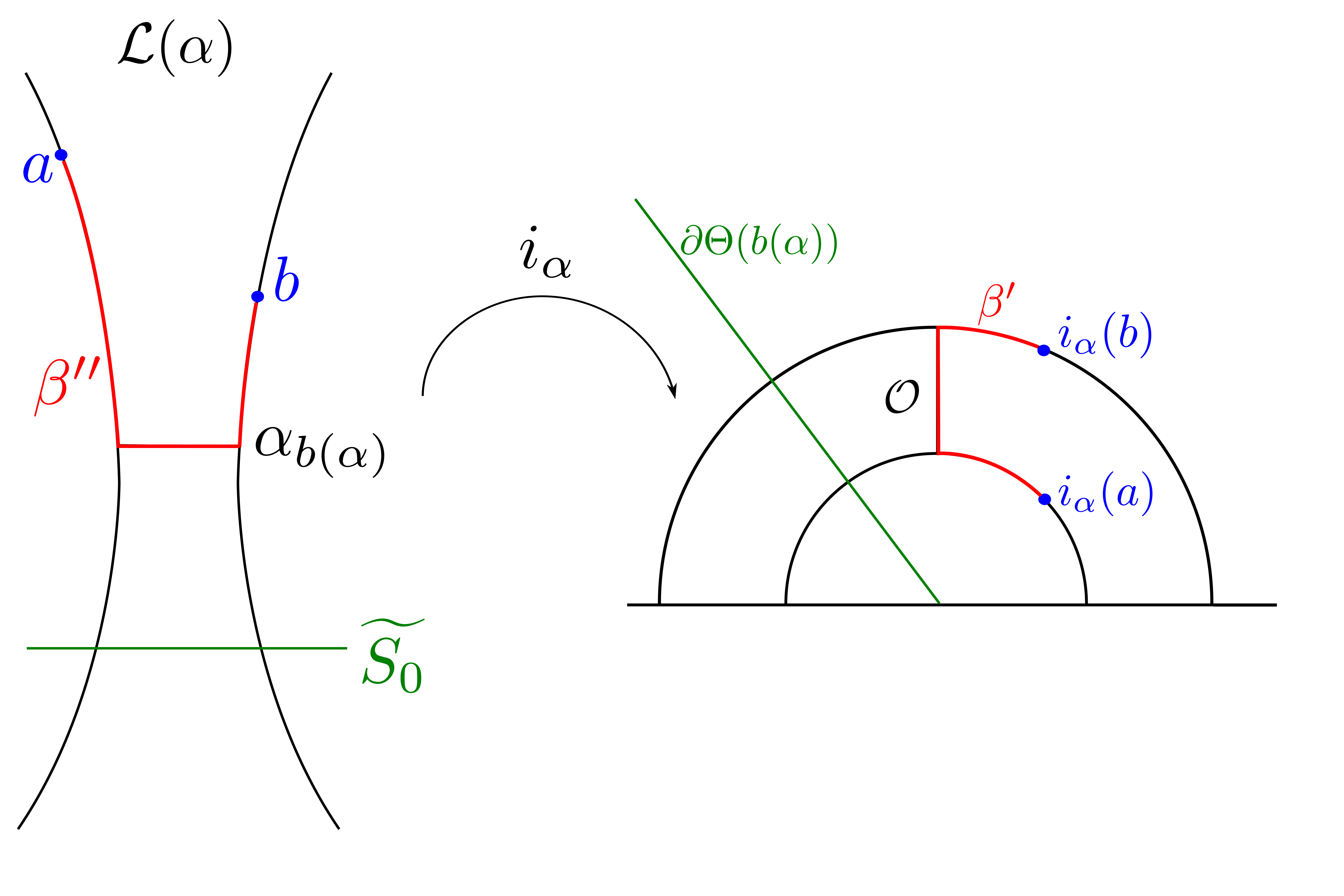
\caption{The ladder $\ladder(\alpha)$ is $\sqrt 2$-bi-Lipschitz equivalent to a strip in the hyperbolic plane. }
\label{U}
\end{figure}
Let $\bU$ be the upper half plane and $\rho_\bU$ its Poincar\'e metric.  For $\ell>0$, take $\bS(\ell) = \{1\le |z| \le e^\ell\}\cap \bU$ the \emph{strip of width $\ell$}; by direct computation, the map \begin{align*}
i_\alpha: (\ladder (\alpha), d_\ladder) &\to (\bS(\ell_{b(\alpha)}(\alpha)), \rho_\bU)\\
 (x,t) & \mapsto e^x\tanh(t-b(\alpha))+ie^x\sech(t-b(\alpha))
 \end{align*}
is $\sqrt2$-bi-Lipschitz and takes $\alpha \times \{b(\alpha) - t\}$ to the boundary of the $t$-neighborhood of the unique common perpendicular $\mathcal O$ to $\partial \bS(\ell_{b(\alpha)}(\alpha))$.   For $\ell = 0$, $\bS(0) = \{z\in \bU: 0\le \Re(z)\le 1\}$.  Let $a \in \{|z| = 1\}\cap \bS(\ell)$ and $b \in \{|z| = e^\ell\}\cap \bS(\ell)$, and 
\begin{equation}\label{beta'}
\beta' = [a, i] \cdot \mathcal O \cdot [ie^\ell, b] \subset \bU,
\end{equation} 
where $[x,y]$ denotes the $\rho_{\bU}$-geodesic segment joining $x$ to $y$.  It is a standard fact in hyperbolic geometry that  $\beta'$ is an $(K(\ell), 0)$-quasi-geodesic segment, where $K$ can be chosen to depend continuously on $\ell$.  By Theorem \ref{Morse},  \begin{equation}\label{nbhd}
d_{\rho_\bU}^H(\beta', [a,b]) \le n(\ell),
\end{equation}
 where $n(\ell)$ is some decreasing continuous function of $\ell$ that goes to infinity has $\ell$ goes to $0$.  The following is almost immediate from our description of $i_\alpha$ and from the observation in (\ref{nbhd}).

\begin{lemma}\label{Hdist}
Let $\alpha\subset \widetilde{S_0}$ be a saddle connection.  Let $a, b$ be two points on different components of $\partial \ladder(\alpha)$, and $\beta$ be the $d_\ladder$-geodesic segment joining them.  Let $\beta'' = \beta_r\cdot \alpha_h \cdot \beta_l$ be the VH-crossing joining $a$ and $b$.  
There is a universal constant $N$ such that $d_\ladder^H(\beta, \beta'') \le N $. 
\end{lemma}
\begin{proof}
Consider the case that $b(\alpha) = b_\pm(\alpha)$, \ie $\alpha$ is not too short at its balance time.  Then, more explicitly, we have $h = b(\alpha)$ and 
\[\beta'' =  [a,\alpha_{b(\alpha)}(0)]\cdot \alpha_{b(\alpha)}\cdot[\alpha_{b(\alpha)}(\ell_{b(\alpha)}(\alpha)),b].\]  Then $i_\alpha(\beta)$ is a $(\sqrt2,0)$ quasi-geodesic, so by Theorem \ref{Morse}, there is a universal constant $M>0$ such that $d_{\rho_\bU}^H(i_\alpha(\beta), [i_\alpha(a), i_\alpha(b)])\le M$.  Defining $\beta'$  by (\ref{beta'}), using (\ref{nbhd}), and because $\ell_{b(\alpha)}(\alpha)\ge 1$, $d_{\rho_\bU}^H(\beta', [i_\alpha(a),i_\alpha(b)]) \le n(1)$.  The explicit description of $\beta''$ yields $i_\alpha(\beta'') =\beta'$.  Applying $i_\alpha\inverse$ to $\beta'$ and collecting constants, we see that $N = \sqrt2(M+n(1))$ is sufficient, in this case.  

The case that $b(\alpha)\not= b_\pm(\alpha)$ is modeled on the previous paragraph.  However, we have to consider the two relevant alternatives (\ref{3a}) and (\ref{3b}) in Definition \ref{VH-geodesics} to construct $\beta''$.  
The point is that VH-geodesics map to quasi-geodesics $\beta' = i_\alpha(\beta'')$ in the corresponding convex subsets of $\bU$ bounded by bi-infinite geodesics. 
\end{proof}

\begin{definition}\label{normal form} Let $\gamma\subset\widetilde{S_0}$ be a concatenation of saddle connections $\alpha_1\cdot ... \cdot \alpha_n$.  A concatenation $\gamma^{VH} = \beta_1\cdot ...\cdot \beta_n$ of VH-crossings $\beta_i = \beta_{l,i}\cdot \alpha_{h_i}\cdot\beta_{r,i}$ is a \emph{VH-geodesic tracking $\gamma$} if for $i \not = n$,  $\beta_{r,i}$ is a point and $\gamma^{VH}$ joins $\partial \gamma$.
\end{definition}
\noindent It is not hard to see that there is only one VH-gedoesic $\gamma^{VH}$ tracking $\gamma$. 
\begin{lemma}\label{HVgeodesics} Let $\gamma\subset\widetilde{S_0}$ be a concatenation of saddle connections $\alpha_1\cdot ... \cdot \alpha_n$, let $\gamma^{VH}$ the VH-geodesic tracking $\gamma$, and let  $\gamma_\ladder^*$ be the $(\ladder,d_\ladder)$-geodesic joining $\partial \gamma$.  Then $d_\ladder^H(\gamma^{VH}, \gamma_\ladder^*)\le N$.  
\end{lemma}
\begin{proof}The ladder $\ladder(\gamma) = \ladder(\alpha_i)\cup_\partial ... \cup_\partial \ladder(\alpha_n)$ is $\sqrt2$-bi-Lipschitz equivalent to a convex region in the hyperbolic plane bounded by bi-infinite geodesics obtained by gluing together translates by hyperbolic isometries of the maps $i_{\alpha_i}$ along their common boundaries.  It is well known that the image of $\gamma^{VH}$ mimics the behavior of the geodesic segment in $\bU$ joining the image of $\partial \gamma$.  In particular, Definition \ref{normal form} guarantees that $\gamma^{VH}$ does not backtrack.  Collecting constants as in the proof of Lemma \ref{Hdist} yields the lemma.  
\end{proof}

We would now like to understand geodesics in $\ladder_+(\alpha)$, as in Lemma \ref{Hdist}.  There are two cases: either $\alpha$ is balanced in $\widetilde {\model{N}}$, \ie $b(\alpha)\ge 0$, or else $b(\alpha)<0$ and $\alpha$ is balanced in $\widetilde{\model M}\setminus \widetilde {\model N}$.  
\begin{definition}\label{VH-geodesics}
Let $\alpha\subset \widetilde{S_0}$ be a saddle connection, and let $x_s, y_t$ be points on the two different components of $\partial \ladder_+(\alpha)$, so that $s,t\ge 0$.  Say that a VH-path $\gamma = \beta_r\cdot \alpha_h\cdot \beta_l$ joining $x_s$ and $y_t$ is a \emph{VH\textsubscript{$+$}-crossing}, if $\beta_r$ and $\beta_l$ are geodesics parameterized by arclength, $\alpha_h$ is a $d_h$-geodesic segment and
\begin{enumerate}[(i)]
\item if  $b_+(\alpha) \ge 0$, then $\gamma$ is a VH-geodesic.  
\item if $b_+(\alpha) < 0$, then $h = 0$.  
\end{enumerate}
\end{definition}
Given $\hat{b}\in \bR$, let $\Theta(\hat{b}) = \{z: 1>\cos(\arg(z))\ge \tanh(-\hat{b})\}\subset \bU$, so that $\partial \Theta(\hat{b})$ contains a component of $ \partial\cN_{\hat{b}}(\mathcal O)$; if $\hat{b} \ge0$,  then $\Theta(\hat{b})\supset \mathcal O$, while $\Theta(\hat{b})\cap \mathcal O=\emptyset$, otherwise.   Since $i_\alpha$ maps $\alpha_t = (\alpha, t)$ to the boundary of the $t$-neighborhood of $\mathcal O$, we see that 
\[i_\alpha|_{\ladder_+(\alpha)}: \ladder_+(\alpha) \to \bS(\ell_{b(\alpha)}(\alpha))\cap \Theta (b(\alpha))\]
is a $\sqrt2$-bi-Lipschitz diffeomorphism taking  $\widetilde{S_0}\cap \ladder_+(\alpha)$ into $\partial \Theta(b(\alpha))$.   \begin{lemma} \label{Hdistplus1}Suppose $\alpha$ is a saddle connection such that $b_+(\alpha)\ge 0$, and let $a, b$ be points on different components of $\partial \ladder_+(\alpha)$.  Let $\beta$ the be $(\ladder, d_\ladder)$-geodesic joining $a$ and $b$, and let $\beta_+$ be a $(\ladder_+, d_\ladder)$-geodesic joining $a$ and $b$.   There is a constant $N'=N'(\epsilon,S)$ such that $d^H_\ladder(\beta, \beta_+) \le N'$.  

Let $\beta_+''$ bet the VH\textsubscript{$+$}-crossing joining $a$ and $b$.  Then $d^H_\ladder(\beta_+'', \beta_+)\le N'$, as well.
\end{lemma}
\begin{proof}The VH-crossing and VH\textsubscript{$+$}-crossing $\beta''$ joining $a$ and $b$ coincide. Then $i_\alpha(\beta'')\subset \bS(\ell_{b(\alpha)}(\alpha))\cap \Theta (b(\alpha))$ is quasi-geodesic in both $\bS(\ell_{b(\alpha)}(\alpha))\cap \Theta (b(\alpha))$ and $\bS(\ell_{b(\alpha)}(\alpha))$.  The geodesic joining $i_\alpha(a)$ and $i_\alpha(b)$ in  $\bS(\ell_{b(\alpha)}(\alpha))\cap \Theta(b(\alpha))$ has bounded Hausdorff distance from the geodesic $[i_\alpha(a),i_\alpha(b)]\subset \bS(\ell_{b(\alpha)}(\alpha))$. Mapping these geodesics back to $\ladder$ via $i_\alpha\inverse$, we again obtain $(\sqrt2,0)$-quasi-geodesics.  The conclusion follows from the proof of Lemma \ref{Hdist} and an application of the triangle inequality.  
\end{proof}
\begin{definition}\label{normal form +} Let $\gamma\subset\widetilde{S_0}$ be a concatenation of saddle connections $\alpha_1\cdot ... \cdot \alpha_n$.  A concatenation $\gamma_+^{VH} = \beta_1\cdot ...\cdot \beta_n$ of VH\textsubscript{$+$}-crossings $\beta_i = \beta_{l,i}\cdot \alpha_{h_i}\cdot\beta_{r,i}$ is a \emph{VH\textsubscript{$+$}-geodesic tracking $\gamma$} if for $i \not = n$,  $\beta_{r,i}$ is a point and $\gamma_+^{VH}$ joins $\partial \gamma$.
\end{definition}
\noindent Again,  $\gamma_+^{VH}$ is uniquely determined by these conditions and $\gamma$.  The following is immediate from Definition \ref{normal form +} and Lemma \ref{Hdistplus1}.
\begin{lemma}\label{cor:normal}Suppose that $\gamma\subset\widetilde{S_0}$ is a concatenation of saddle connections $\alpha_1\cdot ... \cdot \alpha_n$, $b_+(\alpha_i)>0$ for each $i$, and $\gamma^{VH} =\beta_1\cdot ...\cdot \beta_n$ is the VH-geodesic tracking $\gamma$.  Then $\gamma^{VH} = \gamma^{VH}_+\subset \ladder_+$ is a uniform $(\ladder,d_\ladder)$-quasi-geodesic and a uniform $(\ladder_+,d_\ladder)$-quasi-geodesic.  In particular,  if $\gamma^*_\ladder$ and $\gamma^*_{\ladder_+}$ are $(\ladder,d_\ladder)$ and $(\ladder_+, d_\ladder)$ geodesics joining $\partial \gamma$, respectively, then $d_\ladder^H(\gamma^{VH}, \gamma_\ladder^*)\le N'$ and $d_\ladder(\gamma^{VH}, \gamma_{\ladder_+}^*)\le N'$.  
\end{lemma}

\begin{figure}[h]
\def\svgwidth{\columnwidth}
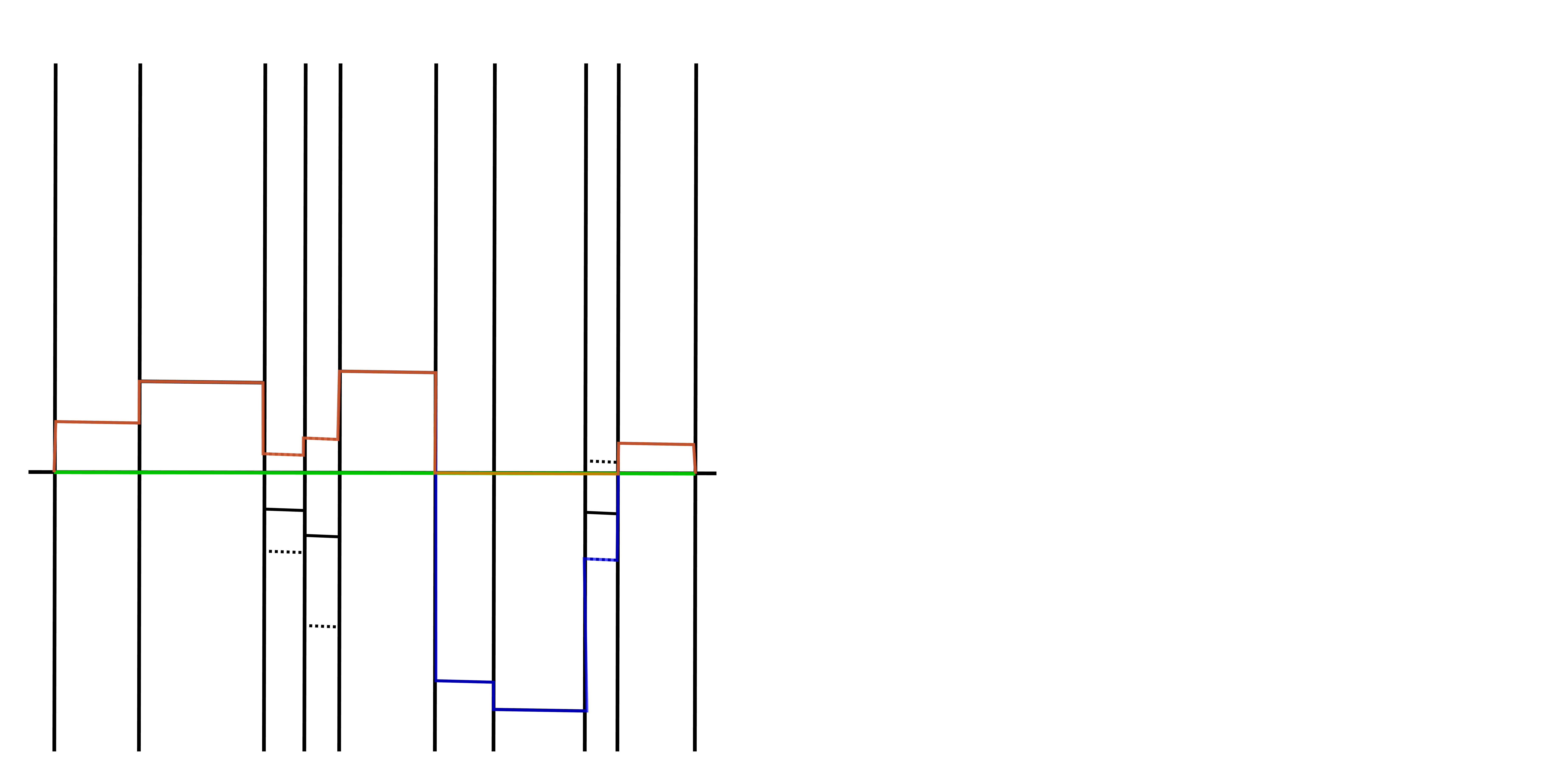
\caption{On the left is $\gamma^{VH}\subset \ladder(\gamma)$ and $\gamma^{VH}_+\subset\ladder_+(\gamma)$; on the right is the corresponding $\sqrt{2}$-bi-Lipschitz convex subset of $\bH^2$.  The partition from the proof of Theorem \ref{almost bounded distance} is $I_1 = [[1,5]]$, $I_2 = [[6,8]]$, and $I_3 = [[9]]$. }
\label{V}
\end{figure}

Now we model the situation that the saddle connection $\alpha$ is balanced in $\widetilde{\model M}\setminus\widetilde{\model N}$.  If $\hat b< 0$, let $\mathcal O' = \partial \Theta(\hat b) \cap \bS(\ell)$; $\mathcal O'$ has $\partial \mathcal O' = \{p,q\}\subset \partial \bS(\ell)$.  Take \begin{equation} \beta_+' = [a,p]\cdot \mathcal O'\cdot [q,b] \subset \bU. \label{beta'+}\end{equation}

\begin{lemma}\label{Hdistplus2}
Let $\alpha\subset S_0$ be a saddle connection such that $b_+(\alpha)< 0$, let $a, b$ be two points on different components of $\partial \ladder_+(\alpha)$, and let  $\beta_+$ be a $(\ladder_+,d_\ladder)$-geodesic segment joining them.  Swapping $a$ and $b$ if necessary, let  $\beta_+''$ be the VH\textsubscript{$+$}-crossing \[\beta_+'' = [a,\alpha_{0}(0)]\cdot \alpha_{0}\cdot[\alpha_{0}(\ell_{0}(\alpha)),b].\] There is a  universal constant $N''$ such that $d_\ladder^H(\beta_+, \beta_+'') \le N''$.
\end{lemma}

\begin{proof}
This is directly analogous to the proof of Lemma \ref{Hdist} when $\alpha$ is not short at balance time.  Define $\beta_+'$ by (\ref{beta'+}).  The key point is that $\beta_+' = i_\alpha(\beta_+'')$, and $\beta_+'$ is uniformly quasi-geodesic in the path metric on $\Theta(b(\alpha))\cap \bS(\ell_{b(\alpha)}(\alpha))$ (which is no longer convex in $\bU$).  This is independent of whether $\alpha$ is short at balance time, because $b_+(\alpha)<0$ by assumption. 
\end{proof}

\noindent We are now in a position to prove the main result from this section.  
\begin{proof}[Proof of Theorem \ref{almost bounded distance}]

By our choice of $p\in \Sigma$ and fixed lift $\tilde p$, the lift $\gamma\subset \widetilde{S_0}$ based at $\tilde p$ of any $d_0$-geodesic loop $\overline\gamma \subset S_0$ based at $p$ decomposes as a concatenation of saddle connections, \ie $\gamma = \alpha_1\cdot ... \cdot \alpha_n \subset \widetilde{S_0}$.  Now, build the ladders $\ladder_+(\gamma)\subset \ladder(\gamma)\subset \widetilde{\model M}$.  The ladders decompose as the union $\ladder(\gamma) = \ladder(\alpha_1)\cup_\partial ... \cup_\partial \ladder(\alpha_n)$.  Let $\gamma^*_\ladder$ be a $(\ladder, d_\ladder)$-geodesic and $\gamma^*_{\ladder_+}$ a $(\ladder_+, d_\ladder)$-geodesic.  

For integers $i\le j$, let $[[i,j]] = [ i,j]\cap \bN$.  Starting with $i_1 = 1$, we iteratively construct a partition $1 = i_1 \le t_1 < i_2\le t_2 < ... <i_k \le t_k = n$ so that the subintervals $I_j=[[i_j,t_j]]$ satisfy the following three properties:
\begin{enumerate}
\item If $b_+(\alpha_{i_j})\ge 0$, then $b_+(\alpha_m)\ge 0$ for each $m\in I_j$, and either $t_j = n$ or $b_+(\alpha_{t_j+1})<0$. \label{above}
\item If $b_+(\alpha_{i_j})<0$, then $b_-(\alpha_m)< 0$ for each $m\in I_j$ and either $t_j = n$ or $b_-(\alpha_{t_j+1})\ge0$. \label{below}
\item For $j\ge 1$, $i_{j+1} = t_j +1$.   
\end{enumerate}
Let $\alpha_{I_j} = \alpha_{i_j}\cdot ... \cdot \alpha_{t_j}$ or $\alpha_{I_j} = \alpha_{i_j}$ if $i_j = t_j$, and set $\ladder(\alpha_{I_j}) = \ladder(\alpha_{i_j})\cup_\partial ... \cup_\partial \ladder(\alpha_{t_j})$ and $\ladder_+(\alpha_{I_j}) = \ladder(\alpha_{I_j}) \cap \ladder_+(\gamma)$.  

\begin{claim}\label{main claim}
There exists a number $D = D(\epsilon, S)$ such that for each $j = 1, ..., k$, we have \[\gamma_{\ladder_+}^*\cap \ladder(\alpha_{I_j}) \subset\cN_{D}((\gamma_\ladder^*\cap \ladder(\alpha_{I_j}))\cup \alpha_{I_j}).\]  That is, restricted to each subladder $\ladder(\alpha_{I_j})$, $\gamma_{\ladder_+}^*$ is close to $\gamma_\ladder^*$ or $\alpha_{I_j}\subset \widetilde{S_0}$.  Moreover, 
$\gamma_{\ladder}^*\cap \ladder_+(\alpha_{I_j})\subset \cN_D(\gamma_{\ladder_+}^*)$. 
\end{claim}
We perform an inductive argument on the number $k$ of subintervals in this partition.  We will prove the claim on the ladder $\ladder(\alpha_{I_1})$ induced by the first interval and show that we may repeat the argument on $d_\ladder$-geodesics joining $\partial (\alpha_{t_1+1}\cdot ... \cdot \alpha_n)$.   A reformulation of the claim says that $d_\ladder$-geodesics satisfy the conclusion of Theorem \ref{almost bounded distance}.  Since  $d_\ladder$-geodesics approximate $d_{\model M}$ and $d_{\model N}$ geodesics within bounded Hausdorff distance by Proposition \ref{quasi-convex}, after we prove Claim \ref{main claim} and collect constants, the theorem will be proved.  In what follows, `uniformly close,'  or `uniformly bounded distance,' etc. means that there is a constant depending at most on $S$ and $\epsilon$, such that the objects mentioned are distance within that constant of each other.  

\begin{proof}[Proof of Claim \ref{main claim}]
Suppose that $I_1$ satisfies (\ref{above}).  Then either $t_1 = n$ or $b_+(\alpha_{t_1+1})<0$.  If $t_1 = n$, then by Lemma \ref{cor:normal}, $d^H_\ladder (\gamma_\ladder^*, \gamma_{\ladder_+}^*)\le 2N'$.  If $b_+(\alpha_{t_1+1})<0$, then VH-geodesic $\gamma^{VH}$ tracking $\gamma$ contains the vertical segment $\ladder(\alpha_{t_1})\cap \ladder(\alpha_{t_1+1})$ joining $0<b_+(\alpha_{t_1})$ to $b_+(\alpha_{t_1+1})<0$, so $\gamma^{VH}$ intersects the terminal endpoint $p_{t_1}$ of $\alpha_{t_1}$.  Then by Lemma \ref{HVgeodesics},  there is a point $q_{t_1}\in \gamma_\ladder^*$ passing uniformly close to $p_{t_1}$.  Since $b_+(\alpha_{t_1+1})<0$, each component of $\partial \ladder_+(\alpha_{t_1+1})$ is union of $1$-separated convex sets in a $\delta'$-hyperbolic space.  The points $\partial \alpha_{t_1+1}$ (nearly) achieve the distance between the two components of $\partial \ladder_+(\alpha_{t_1+1})$ by our description of $\ladder_+(\alpha_{t_1+1})$ in Lemma \ref{Hdistplus2}, so a point $ q_{t_1,+}\in \gamma_{\ladder_+}^*$ comes uniformly close to $\partial \alpha_{t_1+1}$, hence close to $p_{t_1}$.  By hyperbolicity of ladders, the $(\ladder,d_\ladder)$-geodesic and $(\ladder_+,d_\ladder)$-geodesic segments joining $\partial (\alpha_1\cdot ... \cdot \alpha_{t_1})$ are uniformly close to the $(\ladder,d_\ladder)$-geodesic and $(\ladder_+,d_\ladder)$ segments joining $\gamma(0)$ with $q_{t_1}$ and $q_{t_1,+}$, respectively.  Another application of Lemma \ref{cor:normal} in this case therefore proves the theorem on $\ladder(\alpha_{I_1})$ when $I_1$ satisfies (\ref{above}).

We now consider the case that $I_1$ satisfies (\ref{below}).  We will see that the case that $t_1 = n$ follows similarly as in the case that $t_1\not = n$ and  $b_-(\alpha_{t_1+1})\ge0$.  The VH\textsubscript{$+$}-geodesic $\gamma^{VH}_+$ that tracks $\gamma$ on $\alpha_{I_1}$ is equal to $\alpha_{I_1}$, by construction in Definition \ref{normal form +}.  In the spirit of Lemma \ref{cor:normal}, we may apply Lemma \ref{Hdistplus2} to concatenations of saddle connections to see that $\gamma^{VH}_+$ is within bounded Hausdorff distance of $\gamma_{\ladder_+}^*$ on $\ladder(\alpha_{I_1})$.  By the construction of Definition \ref{VH-geodesics} and the conditions of (\ref{below}), the VH-geodesic $\gamma^{VH}$ that tracks $\gamma$ meets $\ladder_+(\alpha_{I_1})$ only in its initial and terminal vertical segments.  The initial segment meets $\ladder_+(\alpha_{I_1})$ exactly at $\alpha_1(0)$ and the terminal vertical segment of $\gamma^{VH}$ meets $\ladder(\alpha_{t_i})\cap \ladder(\alpha_{t_i+1})$ in a segment containing $[0, b_-(\alpha_{t_1+1})]$; in particular $\ladder_+(\alpha_{I_1})\cap \gamma^{VH} \supset \partial \alpha_{I_1}$.   By Lemma \ref{HVgeodesics}, we can now conclude that $\gamma_{\ladder}^*$ is within uniformly bounded Hausdorff distance from this VH-geodesic on $\ladder(\alpha_I)$.  Thus,  $\gamma_{\ladder}^*$ can only be far from $\gamma_{\ladder_+}^*$ in $\ladder(\alpha_{I_1})$ when $I_1$ satisfies (\ref{below}) and $\gamma_\ladder^*\subset \ladder(\alpha_{I_1})\setminus\ladder_+(\alpha_{I_1})$.  

In both cases, since $\gamma_\ladder^*$ and $\gamma_{\ladder^+}^*$ both come close to the initial point of $\alpha_{t_1+1}$, we can now replace $\gamma$ with $\alpha_{t_1+1} \cdot ... \cdot\alpha_{n}$ and repeat the argument, as promised.  
\end{proof}
\noindent This concludes the proof of the theorem.   
\end{proof}

\section{Zeroing out half the manifold}\label{zero out half the manifold}

Let us recall the notation and set up from Section \ref{volume preserving limit maps}.  We have $\nu = (\lambda_{\scriptscriptstyle-}, \lambda_+)$ where $\lambda_\pm$ have bounded geometry, and $Y_0\in \teich (S)$ was a point on the Teichm\"uller geodesic joining $\lambda_-$ and $\lambda_+$.  We defined $M_{\scriptscriptstyle-} = M_{(\lambda_{\scriptscriptstyle-}, Y_0)}$  and constructed a map $\Phi_{\scriptscriptstyle-}: \core(M_{\scriptscriptstyle-}) \to M_\nu$ inducing $\rho(\nu)\circ \rho(\lambda_{\scriptscriptstyle-}, Y_0)\inverse$ on fundamental groups, that was a $B$-bi-Lipschitz embedding, and that was volume preserving away from a compact product region $\mathcal K$.  Take $\inj(M_{\nu})=\epsilon>0$. First we would like to use our results from Section \ref{ladders}.

\begin{lemma}\label{almost bounded distance 2}
There exists a constant $D'>0$ depending only on $\epsilon$, $S$, and $B$ such that the following holds.
Let $\gamma: I \to M_{\scriptscriptstyle -}$ be a geodesic loop based at $p\in \core(M_{\scriptscriptstyle -})$ and  $\widetilde{\gamma}: I \to \widetilde{M_{\scriptscriptstyle-}}$ be a choice of lift.  Let $\Pi: \widetilde {M_{\nu}} \to \im \tilde\Phi_{-}(\widetilde{\gamma})^*$ be the nearest point projection.  

If $d(\tilde\Phi_{\scriptscriptstyle -}(\widetilde\gamma(t)), \Pi(\tilde\Phi_{\scriptscriptstyle -}(\widetilde\gamma(t)))>D'$, then $\Pi(\tilde\Phi_{\scriptscriptstyle -}(\widetilde\gamma(t)))\in \mathcal N _{D'}(\widetilde{M_{\nu}}\setminus \tilde\Phi_{\scriptscriptstyle -}(\widetilde{\core(M_{\scriptscriptstyle -})}))$.  
\end{lemma}
\begin{proof}
Recall that $\widetilde{\Psi}_\nu: \widetilde{\model M_\nu} \to \widetilde{M_\nu}$ is a quasi-isometry by Theorem \ref{model map}.  So given a geodesic segment $\widetilde{\gamma}'\subset \widetilde{\model M_\nu}$, we have that $ \widetilde{\Psi}_\nu(\widetilde{\gamma}')$ is Hausdorff-close to $\widetilde{\Psi}_\nu(\widetilde{\gamma}')^*$, by the Morse Lemma (Theorem \ref{Morse}).  So statements that are true about geodesics in the model manifold $\widetilde{\model M_\nu}$ are true (within bounded Hausdorff distance) in $\widetilde{M_\nu}$, via $\widetilde{\Psi}_\nu$.  Analogously, statements about geodesics in $\widetilde{\core(M_-)}$ are approximated by statements about geodesics in $\widetilde{\model N}$ via $\widetilde{\Psi}_0: \widetilde {\model N} \to \widetilde{\core(M_-)}$.   This means that the conclusion of the lemma will be immediate from Theorem \ref{almost bounded distance} once we establish that $\Phi_- \circ \Psi_0: \model N \to M_\nu$ and $\Psi_\nu|_{\model N}: \model N \to M_\nu$ are good enough approximations of each other.  That is, the conclusion of the lemma is immediate once we  show that
\begin{equation}\label{bounded maps}
\sup_{x\in \model N}d_{M_\nu}(\Phi_{\scriptscriptstyle -}(\Psi_0(x)), \Psi_\nu(x))<\infty.
\end{equation}
To this end, for a negative integer $n$ and level surface $S_n\subset \model N$, the argument given in the paragraph directly preceding (\ref{surface distance}) in the proof of Proposition \ref{volume preserving limit} shows that there is a constant $E = E(S, \epsilon, B)$ such that \[d_{M_\nu}^H( \Psi_\nu(S_{n}),\Phi_{\scriptscriptstyle -}(\Psi_0(S_{n})))\le E.\]
Now, every $x\in \model N$ is within $1/2$ of some level surface $S_n$, and $\widetilde {\Psi_\nu}$ and $\widetilde{\Psi_0}$ are $(L,c)$ quasi-isometries, so $d_{M_{\scriptscriptstyle - }}(\Psi_0(S_n), \Psi_0(x))\le L/2+c$ and $d_{M_{\nu}}(\Psi_\nu(S_n), \Psi_\nu(x))\le L/2+c$.  Finally, $\Phi_{\scriptscriptstyle -}$ is $B$-bi-Lipschitz, and the diameter of an $\epsilon$-thick hyperbolic surface is at most $D = D(\epsilon, S)$, so \[d_{M_\nu}(\Phi_{\scriptscriptstyle -}(\Psi_0(x)), \Psi_\nu(x))\le (\frac L2+c)(B+1)+D+E<\infty,\]
which establishes (\ref{bounded maps}) and completes the proof of the lemma. 
\end{proof}

Let $\mathcal K' = \mathcal K \cup \cN_{BD'}(\partial\core(M_{\scriptscriptstyle-}))$, and let $g_{\scriptscriptstyle-}: \core(M_{\scriptscriptstyle-}) \to \bR$ be a bump function taking values $1$ away from $\cN_1(\mathcal K')$ and $0$ on $\mathcal K'$.  Take \begin{equation}\label{definition}
f_{\scriptscriptstyle-} (x)= \begin{cases} g_{\scriptscriptstyle-}( \Phi_{\scriptscriptstyle-}\inverse(x)), & x\in \im\Phi_{\scriptscriptstyle-}\inverse \\ 0, & \text{else}.\end{cases}
\end{equation}
\begin{remark}\label{definition of f_-}We have constructed $f_{\scriptscriptstyle-}$ so that $\supp(f_{\scriptscriptstyle-})$ is contained in the complement of $ \mathcal N _{D'}(M_{\nu}\setminus \Phi_{\scriptscriptstyle -}(\core(M_{\scriptscriptstyle -})))$.
\end{remark}
The following is one of the main technical ingredients that will go into the proof of Theorem \ref{relation 1} in the next section.  We want to show that by `zeroing out' an end of the doubly degenerate manifold with $f_{\scriptscriptstyle-}$, the bounded fundamental class of the singly degenerate manifold with positive $f_{\scriptscriptstyle-}$-volume survives.  We write $[\widehat{\omega_{\scriptscriptstyle-}}]$ to denote the bounded fundamental class of $M_{\scriptscriptstyle-}$ and $[\widehat{\omega}]$ for that of $M_\nu$.

\begin{prop}\label{zero out half}
With notation as above, we have an equality 
\[[\widehat{\omega_{\scriptscriptstyle-}}]= [\Phi_{\scriptscriptstyle-}^*\widehat{f_{\scriptscriptstyle-}\omega}]\in \Hb_{\bb}^3(M_{\scriptscriptstyle-};\bR)\]
This equality takes the form \[ [\widehat{\omega_{\scriptscriptstyle-}}]=[\widehat{f_{\scriptscriptstyle-}\omega}]\in \Hb_{\bb}^3(S;\bR)\]
when we suppress markings, \ie pass the first equality through the isometric isomorphism $\Hb_{\bb}^3(M_{\scriptscriptstyle-};\bR)\to \Hb_{\bb}^3(S;\bR)$.  

With $f_+$ and $M_+ = M_{(Y_0, \lambda_+)}$ defined analogously, \[ [\widehat{\omega_+}]=[\widehat{f_+\omega}]\in \Hb_{\bb}^3(S;\bR).\] 

\end{prop}

\begin{proof}
By Lemma \ref{point support}, $[\widehat{\omega_{\scriptscriptstyle-}}]$ is represented by the cocycle $\widehat{\omega_{{\scriptscriptstyle-},p}}$, where $p \in \core(M_{\scriptscriptstyle-})$.  
By Proposition \ref{compact support},  $[\widehat{\omega_{\scriptscriptstyle-,p}}]$ is represented by the cocycle $\widehat{g_{\scriptscriptstyle-}\omega_{-,p}}$, because $1- g_{\scriptscriptstyle-}$ has compact support on $\core(M_{\scriptscriptstyle-})$.
\begin{assumption}\label{notation}
Assume that $\sigma : \Delta_3 \to M_{\scriptscriptstyle-}$ maps the vertices of $\Delta_3$ to $p$.  Then $\str_p\sigma = \str\sigma$, and we drop the subscript ${}_p$ from notation. 
\end{assumption}
Restricting to $\im \Phi_{\scriptscriptstyle-}$ and appealing to the definition of $f_{\scriptscriptstyle-}$ in (\ref{definition}), we compute \[{\Phi_{\scriptscriptstyle-}\inverse}^*(g_{\scriptscriptstyle-}\omega_{\scriptscriptstyle-}) = f_{\scriptscriptstyle-} \cdot \Jac(\Phi_{\scriptscriptstyle-}\inverse)\cdot \omega = f_{\scriptscriptstyle-}\omega,\] because $\Phi_{\scriptscriptstyle-}\inverse$ is orientation and  volume preserving on the support of $f_{\scriptscriptstyle-}$ by Proposition \ref{volume preserving limit}.

We therefore have have \begin{align}
(\widehat{g_{\scriptscriptstyle-}\omega_{-}} - \Phi_{\scriptscriptstyle-}^*\widehat{f_{\scriptscriptstyle-}\omega})(\sigma) & = \int_{\str_{\scriptscriptstyle-}\sigma}g_{\scriptscriptstyle-}\omega_{\scriptscriptstyle-} - \int_{\str \Phi_{\scriptscriptstyle-*}\sigma}f_{\scriptscriptstyle-}\omega \\
 & = \int_{\Phi_{\scriptscriptstyle-*}\str_{\scriptscriptstyle-}\sigma}{\Phi_{\scriptscriptstyle-}\inverse}^*(g_{\scriptscriptstyle-}\omega_{\scriptscriptstyle-}) - \int_{\str \Phi_{\scriptscriptstyle-*}\sigma}f_{\scriptscriptstyle-}\omega \\
 \label{final form}
 &  = \int_{\Phi_{\scriptscriptstyle-*}\str_{\scriptscriptstyle-}\sigma- \str {\Phi_{\scriptscriptstyle-*}} \sigma}f_{\scriptscriptstyle-}\omega
\end{align}
We define $H(\tau)$ \emph{exactly} as in the discussion preceding Lemma \ref{coboundary lemma}, and take 
\[C_{f_{\scriptscriptstyle-}}(\tau) = \int_{\Delta_2\times I }H(\tau)^*(f_{\scriptscriptstyle-}\omega).\]

Apply the proof of Lemma \ref{coboundary lemma}, using the equality (\ref{final form}) to see that \[ (\widehat{g_{\scriptscriptstyle-}\omega_{-}} - \Phi_{\scriptscriptstyle-}^*\widehat{f_{\scriptscriptstyle-}\omega})(\sigma) = dC_{f_{\scriptscriptstyle-}}(\sigma).\] 

As in Assumption \ref{notation}, assume $\tau$ maps the vertices of $\Delta_2$ to $p$.  We want to show that $|C_{f_{\scriptscriptstyle-}}(\tau)|<M$, for some $M$.  We mimic the proof of Proposition \ref{coboundary bounded}.  The proof had two main ingredients, namely (1) that trajectories $H(\tau)(x, I)$ had uniformly bounded length and (2) that the area of the level triangle $H(\tau)(\Delta_2, t)$ was uniformly bounded.    We proceed \emph{exactly} as in the proof of Lemma \ref{coboundary bounded}, except we need to replace the bound on trajectories with something that makes more sense.  We have
\begin{align}
\left| \int_{\Delta_2\times I} H(\tau)^*f_{\scriptscriptstyle-}\omega \right| & = \left|\int_0^1  \int_{\Delta_t} f_{\scriptscriptstyle-}\left\langle \frac{\partial H(\tau)}{\partial t} , \vec n_t \right\rangle ~dA_t~dt \right|\\ 
&\le  \int_0^{\frac12} \int_{\Delta_t} \left|\sqrt{f_{\scriptscriptstyle-}} \left\langle \frac{\partial H(\tau)}{\partial t} , \vec n_t \right\rangle\right| \sqrt{f_{\scriptscriptstyle-}}~dA_t~dt.\label{integral inequality 2}
\end{align}

The geodesic triangle $\str{\Phi_{\scriptscriptstyle-*}}\tau$ tracks the B-bi-Lipschitz triangle $\Phi_{-*}\str \tau$ on $\supp(f_-)$ (see Figure \ref{zero out}).  
  \begin{figure}[h]
\includegraphics[width=4in]{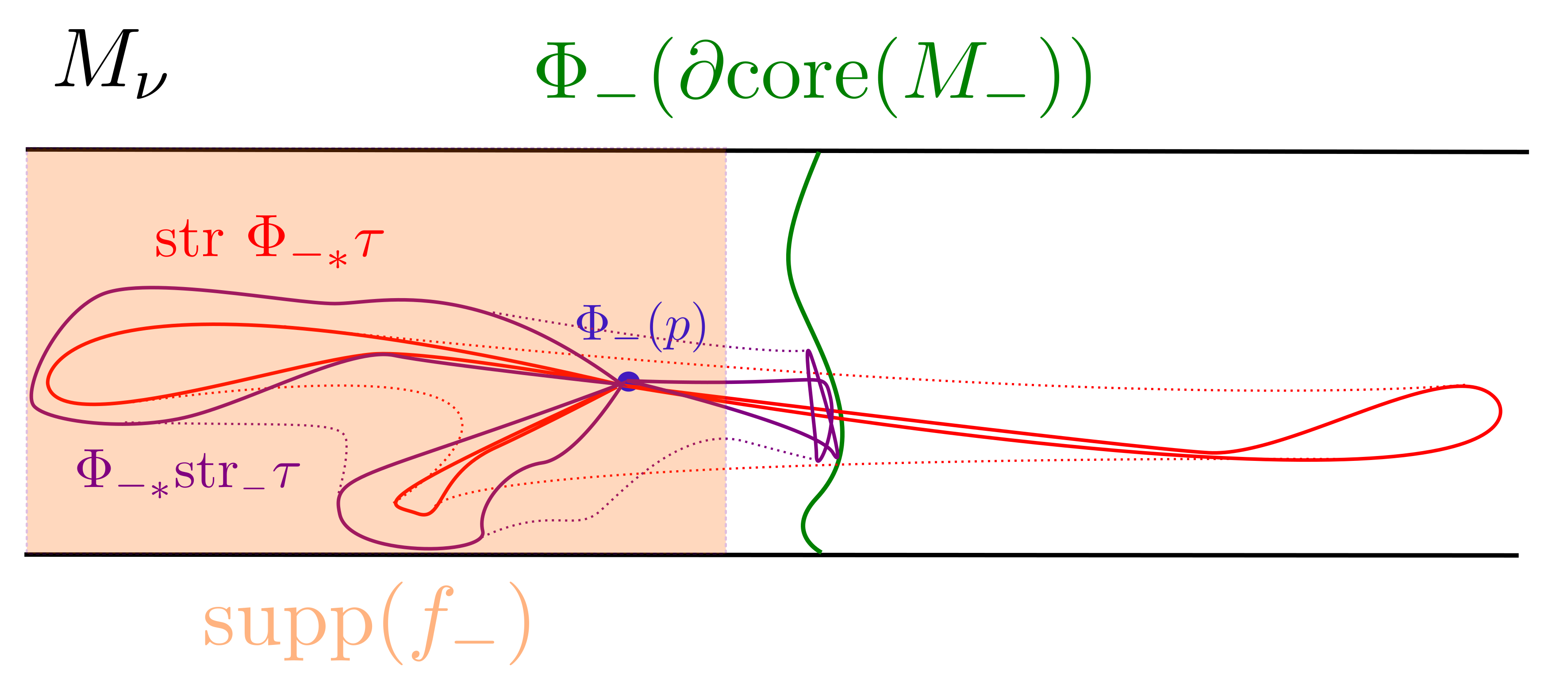}
\caption{Trajectories of $H(\tau)$ have bounded length when intersected with the support of $f_-$,  even though they are generally unbounded in $M_\nu$.}
\label{zero out}
\end{figure}
 More precisely, by Lemma \ref{almost bounded distance 2} and by construction of $f_-$ (see Remark \ref{definition of f_-}), there is a $D'$ such that if $p\in [i,j]\subset\Delta_2$ and  $q = \widetilde{\Phi_{\scriptscriptstyle-*}\str_{\scriptscriptstyle -} \tau}(p)\in \supp(f_-)$, then
 \[d_{\widetilde{M}}(\Pi(q),q)\le D',\] 
 where $\Pi: \bH^3 \to \widetilde{\str_{\scriptscriptstyle } \Phi_{\scriptscriptstyle-*}\tau}([i,j])$ is closest point projection.  Take $c:=(2B\log\sqrt3+2D' +1)$; then
 \begin{equation}\label{traj2}
 \left|\sqrt{f_{\scriptscriptstyle-}} \left\langle \frac{\partial H(\tau)}{\partial t} , \vec n_t \right\rangle\right| \le 2c.
 \end{equation}
 Compare (\ref{traj2}) with (\ref{traj1}) - (\ref{bounded tracks}) in the proof of Proposition \ref{coboundary bounded}.

Recall that $\Phi_{\scriptscriptstyle-}$ is $B$-Lipschitz. On $\supp(f_-)$, the discussion preceding estimate (\ref{bounded area}) bounding the area of any level triangle goes through as before
\begin{equation} \label{bounded area 2}
\int_{\Delta_t}\sqrt{f_{\scriptscriptstyle-}}~dA_t \le  \pi (Bc)^2, ~t\in [0, 1/2].
\end{equation}

 Therefore, 
 \begin{align*}
  | C_{f_{\scriptscriptstyle-}}(\tau)| = \left| \int_{\Delta_2\times I} H(\tau)^*f_{\scriptscriptstyle-}\omega \right|  \le \pi B^2c^3,
  \end{align*}
 which completes the proof of the proposition.
\end{proof}

\section{The relations and a Banach subspace}\label{relations and banach subspace}
Let $\lambda, \lambda' \in \cEL_{\bb}(S)$ and $X, Y\in \teich (S)$, then we have representations $\rho_{(\nu_{\scriptscriptstyle-},\nu_+)}: \pi_1(S)\to \G$ where $\nu_\pm \in\{\lambda, \lambda', X, Y\}$.  Let $\widehat{\omega}({\nu_{\scriptscriptstyle-}},{\nu_+}) \in \Cb_{\bb}^3(S;\bR)$ be the corresponding bounded $3$-cocycle.  
We gather the results from the previous sections to conclude that 
\begin{theorem}\label{main 1}
Let $\lambda, \lambda' \in \cEL(S)$ have bounded geometry and $X, Y\in \teich (S)$ be arbitrary.  We have an equality in bounded cohomology \[[\widehat{\omega}(\lambda', \lambda)]= [\widehat{\omega}(X, \lambda)]+ [\widehat{\omega}(\lambda',Y)] \in \Hb_{\bb}^3(S;\bR).\]
\end{theorem}

\begin{proof}
The constant function $1$ on $M_{(\lambda,\lambda')}$ decomposes as a sum $1=f_-+f_+$ where $f_\pm\ge 0$ and  $\{x\in M_{(\lambda',\lambda)}: (f_-\cdot f_+)(x)\not= 0\}$ is precompact.   By linearity of the integral, $[\widehat{f_+\omega}(\lambda', \lambda)]+[\widehat{f_-\omega}(\lambda', \lambda)]  = [\widehat{\omega}(\lambda', \lambda)]$.  Consider the sequence $\{Y_i\}_{i\in \bZ}$ along the Teichm\"uller geodesic $t\mapsto Y_t$ between $\lambda'$ and $\lambda$. 
Applying Proposition \ref{zero out half}, \[[\widehat{f_-\omega}(\lambda', \lambda)] =  [\widehat{\omega}(\lambda',Y_{k_-})], \] 
and \[[\widehat{f_+\omega}(\lambda', \lambda)] =  [\widehat{\omega}(Y_{k_+}, \lambda)],\]
for some $k_-, k_+\in \bZ$.
By Theorem \ref{main}, $[\widehat{\omega}(Y_{k_+}, \lambda)] = [\widehat{\omega}(X, \lambda)]$  and $[\widehat{\omega}(\lambda',Y_{k_-})] = [\widehat{\omega}(\lambda',Y)]  $.  Stringing together the equalities, we arrive at the claim.  
\end{proof}

\begin{cor}\label{cor 1}
If $\lambda_1, \lambda_2, \lambda_3 \in \cEL_{\bb}(S)$ are distinct,  then we have an equality in bounded cohomology \[[\widehat{\omega}(\lambda_1, \lambda_3)]= [\widehat{\omega}(\lambda_1, \lambda_2)]+ [\widehat{\omega}(\lambda_2,\lambda_3 )] \in \Hb_{\bb}^3(S; \bR).\]
\end{cor}
\begin{proof}
By Theorem \ref{main 1} for any $X\in \teich (S)$,  \begin{equation}\label{split up} 
[\widehat{\omega}(\lambda_1, \lambda_2)]+ [\widehat{\omega}(\lambda_2,\lambda_3 )] =  [\widehat{\omega}(\lambda_1, X)] +[\widehat{\omega}(X, \lambda_2)]+ [\widehat{\omega}(\lambda_2,X )]+ [\widehat{\omega}(X,\lambda_3 )].
\end{equation}
There is an orientation reversing isometry $M_{(X,\lambda_2)}\to M_{(\lambda_2,X)}$ respecting  markings, so $[\widehat{\omega}(X, \lambda_2)] = - [\widehat{\omega}(\lambda_2,X )]$.  The right hand side of (\ref{split up}) becomes \[ [\widehat{\omega}(\lambda_1, X)] + [\widehat{\omega}(X,\lambda_3 )] = [\widehat{\omega}(\lambda_1, \lambda_3)],\]
by another application of Theorem \ref{main 1}.
\end{proof}

The following is an easy consequence of Theorem \ref{old work} (see \cite[Theorem 6.2 and Theorem 7.7]{Farre:bdd}).
\begin{theorem}\label{infinity norm}
Suppose $\{\lambda_n\}_{n\in \bN}\subset \cEL(S)$ are distinct, and suppose $\{\alpha_n\}_{n\in \bN}\subset \bR$ is such that $\sum_{n = 1}^\infty \alpha_n[\widehat{\omega}(X, \lambda_n)] \in \Hb_{\bb}^3(S;\bR)$.  Then $ \sum_{n = 1}^\infty \alpha_n[\widehat{\omega}(X, \lambda_n)] = 0 $ if and only if $\alpha_n = 0 $ for all $n$.
\end{theorem}

\begin{proof}The `if' direction is trivial.  For the `only if' direction, we apply Theorem \ref{old work} to see that there is a number $\epsilon_S>0$ depending only on $S$ such that any finite sum 
\[\|\sum_{n = 1}^N \alpha_n[\widehat{\omega}(X, \lambda_n)] \|_\infty \ge \epsilon_S\max\{|\alpha_n|\}.\]
So, if $\alpha_{n_0}\not = 0$ for some $n_0$, then the infinite sum 
\[\|\sum_{n = 1}^\infty \alpha_n[\widehat{\omega}(X, \lambda_n)] \|_\infty \ge \epsilon_S|\alpha_{n_0}| >0.\]
Thus, $\sum_{n = 1}^\infty \alpha_n[\widehat{\omega}(X, \lambda_n)] = 0$ only if $\alpha_{n} = 0$ for all $n$.  
\end{proof}

Let $\cEL_{\bb}(S)\subset \cEL(S)$ be the set of bounded geometry ending laminations; $\cEL_{\bb}(S)$ is a mapping class group invariant subspace.  Fix a base point $X\in \teich(S)$, and define $\iota: \cEL_{\bb}(S) \to \Hb_{\bb}^3(S;\bR)$ by the rule $\iota(\lambda) = [\widehat{\omega}(X, \lambda)]$.  By Theorem \ref{main}, $\iota$ does not depend on the choice of $X$.  By Theorem \ref{main 1},  $[\widehat{\omega}(\lambda', \lambda)] = \iota(\lambda) - \iota(\lambda')$, so every doubly degenerate bounded fundamental class without parabolics is in the $\bR$-linear span of $\im\iota$.  Let $Z\subset \Hb_{\bb}^3(S;\bR)$ be the subspace of zero-seminorm elements.  Then $\overline{\Hb}_{\bb}^3(S;\bR) =\Hb_{\bb}^3(S;\bR)/Z$ is a Banach space with $\| \cdot \|_\infty$ norm.   Take $\bar\iota$ to be the composition of $\iota$ with the quotient $\Hb_{\bb}^3(S;\bR)\to \overline{\Hb}_{\bb}^3(S;\bR)$.  By Theorem \ref{infinity norm}, $\im \bar\iota$ is a topological basis for the $\|\cdot \|_\infty$-closure of the span of its image $\mathcal V\subset \overline{\Hb}_{\bb}^3(S;\bR)$.  By a topological basis, we mean that $\sum_{n = 1}^\infty \alpha_n\bar\iota(\lambda_n) = 0$ only if $\alpha_{n} = 0$ for all $n$ when $\sum_{n = 1}^\infty \alpha_n\bar\iota(\lambda_n) \in \overline{\Hb}_{\bb}^3(S;\bR)$.  We have produced
\begin{cor}
The Banach subspace $\mathcal V\subset \overline{\Hb}_{\bb}^3(S;\bR)$ has topological basis $\bar\iota(\cEL_{\bb}(S))$, $\bar\iota$ is a mapping class group equivariant, and $\mathcal V$ contains all of the doubly degenerate bounded volume classes with bounded geometry. 
\end{cor}

\bibliography{groups}{}
\bibliographystyle{amsalpha.bst}
\Addresses
\end{document}